\ifx \processToSlides \undefined
  \ifx\usepackage\RequirePackage    
    \let\usingAmsArtXII\usepackage  
  \fi
\else
  \documentclass[12pt]{amsart}
  \usepackage{amssymb,amscd}
  \usepackage[letterpaper,margin=2cm,includeheadfoot]{geometry}
  \def \useHugeSize {}
  \def \numberingIsThrough {}
\fi

\ifx \labelsONmargin\undefined

  \ifx\RequirePackage\undefined
    \expandafter\ifx\csname amsart.sty\endcsname\relax
        \documentstyle[12pt,amssymb,amscd]{amsart}
    \fi

    \def\atSign{@@}

    \def\mathbb{\Bbb}
    \def\mathfrak{\frak}
    \def\mathbf{\bold}
    \ifx\boldsymbol\undefined   
      \def\boldsymbol#1{{\bold #1}}
    \fi

    \def\mathbit{\boldsymbol}

    \makeatletter
    \newenvironment{proof}{%
         \@ifnextchar[{%
                       \expandafter\let\expandafter\end@proof
                         \csname endpf*\endcsname
                         \my@proof
                      }{\let\end@proof\endpf\pf}%
        }{\end@proof}
    \def\my@proof[#1]{\@nameuse{pf*}{#1}}
    \makeatother

    \def\xrightarrow[#1]#2{@>{#2}>{#1}>}
    \def\xleftarrow[#1]#2{@<{#2}<{#1}<}

    \def\providecommand#1{\def#1}
    \def\emph#1{{\em #1}}
    \def\textbf#1{{\bf #1}}

    \def\mathring{\overset{\,\,{}_\circ}}
  \else
      \ifx\usepackage\RequirePackage
      \else
        \documentclass[12pt]{amsart}
        \usepackage{amssymb,amscd}
    \let\usingAmsArtXII\usepackage
      \fi
      \ifx \mathring\undefined      
        \DeclareMathAccent{\mathring}{\mathalpha}{operators}{"17}
      \fi

      \long\def\FAKEendPROOF{\endtrivlist}
      \ifx\endproof\FAKEendPROOF
      \def\endproof{\qed\endtrivlist}
      \fi

      \ifx\mathbit\undefined    
        \DeclareMathAlphabet{\mathbit}{OML}{cmm}{b}{it}
      \fi

      \def\atSign{@}

      \ifx \usingAmsArtXII \undefined
      \else             
    \advance\oddsidemargin -0.1\textwidth
    \evensidemargin\oddsidemargin
      \fi

      \def\Sb#1\endSb{_{\substack{#1}}}
      \def\Sp#1\endSp{^{\substack{#1}}}

  \fi
\makeatletter
\@ifundefined{DeclareFontShape}
     {\@ifundefined{selectfont}
             {
                \message{We need AmS-LaTeX, this version of LaTeX does not
                        support it}\@@end
             }{
                \def\mathcal{\cal}
                
                \def\pcyr{%
                        \def\default@family{UWCyr}%
                        \let\oldSl@\sl
                        \def\sl{\def\default@shape{it}\oldSl@}%
                        \cyracc
                        \language\Russian\family{UWCyr}\selectfont
                }
             }}{
                \DeclareFontEncoding{OT2}{}{\noaccents@}
                \DeclareFontSubstitution{OT2}{cmr}{m}{n}
                \DeclareFontFamily{OT2}{cmr}{\hyphenchar\font45 }
                \DeclareFontShape{OT2}{cmr}{m}{n}{%
                     <5><6><7><8><9><10>gen*wncyr %
                     <10.95><12><14.4><17.28><20.74><24.88> wncyr10 %
                }{}
                \DeclareFontShape{OT2}{cmr}{m}{it}{%
                     <5><6><7><8><9><10> gen * wncyi%
                     <10.95><12><14.4><17.28><20.74><24.88> wncyi10%
                }{}
                \DeclareFontShape{OT2}{cmr}{bx}{n}{%
                     <5><6><7><8><9><10> gen * wncyb%
                     <10.95><12><14.4><17.28><20.74><24.88> wncyb10%
                }{}
                \DeclareFontShape{OT2}{cmr}{m}{sl}{%
                     <-> ssub * cmr/m/it%
                }{}
                \DeclareFontShape{OT2}{cmr}{m}{sc}{%
                     <5><6><7><8><9><10>%
                     <10.95><12><14.4><17.28><20.74><24.88> wncysc10%
                }{}
                \DeclareFontFamily{OT2}{cmss}{\hyphenchar\font45 }
                \DeclareFontShape{OT2}{cmss}{m}{n}{%
                     <8><9><10> gen * wncyss%
                     <10.95><12><14.4><17.28><20.74><24.88> wncyss10%
                }{}
                
                \def\cyrencodingdefault{OT2}
                \def\pcyr{%
                        \cyracc
                        \let\encodingdefault\cyrencodingdefault
                        \language\Russian\fontencoding{OT2}\selectfont
                }
             }

\@ifundefined{theorembodyfont}
     {
        \def\theorembodyfont#1{\relax}
        \makeatletter
          \let\@@th@plain\th@plain
          \def\th@plain{ \@@th@plain \slshape }
        \makeatother
        \let\normalshape\relax
     }{}

\ifx\cprime\undefined
     \def\cprime{$'$}
\fi
\makeatletter
  \def\@sect@my#1#2#3#4#5#6[#7]#8{%
\ifnum #2>\c@secnumdepth
   \let\@svsec\@empty
 \else
   \refstepcounter{#1}%
\edef\@svsec{\ifnum#2<\@m
             \@ifundefined{#1name}{}{\csname #1name\endcsname\ }\fi
\noexpand\rom{\csname the#1\endcsname.}\enspace}\fi
 \@tempskipa #5\relax
 \ifdim \@tempskipa>\z@ 
   \begingroup #6\relax
   \@hangfrom{\hskip #3\relax\@svsec}{\interlinepenalty\@M #8\par}%
   \endgroup
   \if@article\else\csname #1mark\endcsname{%
        \ifnum \c@secnumdepth >#2\relax\csname the#1\endcsname. \fi#7}\fi
\ifnum#2>\@m \else
       \let\@tempf\\ \def\\{\protect\\}\addcontentsline{toc}{#1}%
{\ifnum #2>\c@secnumdepth \else
             \protect\numberline{%
               \ifnum#2<\@m
               \@ifundefined{#1name}{}{\csname #1name\endcsname\ }\fi
               \csname the#1\endcsname.}\fi
           #8}\let\\\@tempf
     \fi
 \else
  \def\@svsechd{#6\hskip #3\@svsec
    \@ifnotempty{#8}{\ignorespaces#8\unskip
       \ifnum\spacefactor<1001.\fi}%
        \ifnum#2>\@m \else
          \let\@tempf\\ \def\\{\protect\\}\addcontentsline{toc}{#1}%
            {\ifnum #2>\c@secnumdepth \else
              \protect\numberline{%
                \ifnum#2<\@m
                \@ifundefined{#1name}{}{\csname #1name\endcsname\ }\fi
                \csname the#1\endcsname.}\fi
             #8}\let\\\@tempf\fi}%
 \fi
\@xsect{#5}}
  \ifx\RequirePackage\undefined
  \let\@sect\@sect@my             
  \fi
  \ifx\RequirePackage\undefined 
  \def\th@remark@my{\theorempreskipamount6\p@\@plus6\p@
    \theorempostskipamount\theorempreskipamount
    \def\theorem@headerfont{\it}\normalshape}
    \let\th@remark\th@remark@my
  \else             
    \let\o@@remark\th@remark

    \ifx\theorempostskipamount\undefined        
    \else                       
      \def\th@remark{\o@@remark
    \ifdim\theorempostskipamount < 2pt\relax
      \theorempostskipamount\theorempreskipamount
         \multiply\theorempostskipamount\tw@
         \divide\theorempostskipamount\thr@@
    \fi
      }
    \fi
  \fi
\let\myLabel\@gobble
\def\labelsONmargin{\@mparswitchfalse\def\myLabel##1{\@bsphack\marginpar
                                  {\normalshape\tiny\rm Label ##1}\@esphack}}
\ifx \url\undefined
  \def\url#1{{\tt #1}}%
\fi
\def\cyracc{\def\u##1{
                \if \i##1\char"1A%
                \else \if I##1\char"12%
                \else \accent"24 ##1\fi\fi }%
\def\"##1{\if e##1{\char"1B}%
                \else \if E##1{\char"13}%
                \else \accent"7F ##1\fi\fi }%
\def\9##1{\if##1z\char"19
\else\if##1Z\char"11
\else\if##1E\char"03
\else\if##1e\char"0B
\else\if##1u\char"18
\else\if##1U\char"10
\else\if##1A\char"17
\else\if##1a\char"1F
\else\if##1p\char"7E
\else\if##1P\char"5E
\else\if##1Q\char"5F
\else\if##1q\char"7F
\else\if##1i\char"1A
\else\if##1I\char"12
\else\if##1N\char"7D
\fi
\fi
\fi
\fi
\fi
\fi
\fi
\fi
\fi
\fi
\fi
\fi
\fi
\fi
\fi
}%
\def\cydot{{\kern0pt}}}%
\def\cydot{$\cdot$}

\ifx\Russian\undefined
        \def\Russian{0\relax
    \message{Don't know the hyphenation rules for Russian^^J
                        Please do INITeX with `input  russhyph' in the
                        command line}%
                \gdef\Russian{0\relax}%
        }
\fi
\ifx\RequirePackage\undefined           
  \def\@putname#1#2#3#4{\def\@@ref{#3}\let\old@bf\bf
        \def\bf##1{\old@bf\if?\noexpand##1?{#4}\else##1\fi}%
    #1{#2}%
        \let\bf\old@bf}
\else                       
  \def\@putname#1#2#3#4{\def\@@ref{#3}\let\old@bf\bf    
    \let\old@reset@font\reset@font          
        \def\bf##1{\old@bf\if?\noexpand##1?{#4}\else##1\fi}%
    \def\reset@font##1##2{\old@reset@font##1\if?\noexpand##2?{#4}\else##2\fi}#1{#2}%
        \let\bf\old@bf\let\reset@font\old@reset@font}
\fi
\let\my@ref=\ref
\def\ref#1{\@putname\my@ref{#1}{#1}{\tiny\rm\@@ref}}
\let\my@pageref=\pageref
\def\pageref#1{\@putname\my@pageref{#1}{#1}{\tiny\rm\@@ref}}
\let\my@cite=\cite
\def\cite#1{\@putname\my@cite{#1}{\@citeb}{\tiny\rm\@@ref}}
\makeatother
\fi
\relax
\theoremstyle{plain} 
\theorembodyfont{\sl}

\ifx\useDoubleSpacing\undefined
\else
  \usepackage{setspace}
\fi

\ifx \numberingIsThrough \undefined
\numberwithin{equation}{section}
\fi
\theorembodyfont{\rm}
\theoremstyle{definition}

\ifx \numberingIsThrough \undefined
\newtheorem{definition}{Definition}[section]
\else
\newtheorem{definition}{Definition}
\fi

\newtheorem{conjecture}[definition]{Conjecture}
\newtheorem{example}[definition]{Example}

\theoremstyle{remark}

\newtheorem{remark}[definition]{Remark} 

\ifx \numberingIsThrough \undefined
\newtheorem{note}{Note}[section] 
\newtheorem{summary}{Summary}[section] 
\else

\fi

\theoremstyle{plain} 
\theorembodyfont{\sl}

\newtheorem{theorem}[definition]{Theorem}
\newtheorem{lemma}[definition]{Lemma}
\newtheorem{corollary}[definition]{Corollary}
\newtheorem{proposition}[definition]{Proposition}

\begin{document}
\bibliographystyle{amsplain}

\ifx\useHugeSize\undefined
\else
\Huge
\fi

\relax
\renewcommand{\v}{\varepsilon} \newcommand{\p}{\rho}
\newcommand{\m}{\mu}
\def\im{{\bf im}}
\def\ker{{\bf ker}}
\def\Pic{{\bf Pic}}
\def\re{{\bf re}}
\def\e{{\bf e}}
\def\a{\alpha}
\def\ve{\varepsilon}
\def\b{\beta}
\def\D{\Delta}
\def\d{\delta}
\def\f{{\varphi}}
\def\ga{{\gamma}}
\def\L{\Lambda}
\def\lo{{\bf l}}
\def\s{{\bf s}}
\def\A{{\bf A}}
\def\B{{\bf B}}
\def\cB{{\mathcal {B}}}
\def\C{{\mathbb C}}
\def\F{{\bf F}}
\def\G{{\mathfrak {G}}}
\def\g{{\mathfrak {g}}}
\def\b{{\mathfrak {b}}}
\def\q{{\mathfrak {q}}}
\def\f{{\mathfrak {f}}}
\def\k{{\mathfrak {k}}}
\def\l{{\mathfrak {l}}}
\def\m{{\mathfrak {m}}}
\def\n{{\mathfrak {n}}}
\def\o{{\mathfrak {o}}}
\def\p{{\mathfrak {p}}}
\def\s{{\mathfrak {s}}}
\def\t{{\mathfrak {t}}}
\def\r{{\mathfrak {r}}}
\def\z{{\mathfrak {z}}}
\def\h{{\mathfrak {h}}}
\def\H{{\mathcal {H}}}
\def\O{\Omega}
\def\M{{\mathcal {M}}}
\def\T{{\mathcal {T}}}
\def\N{{\mathcal {N}}}
\def\U{{\mathcal {U}}}
\def\Z{{\mathbb Z}}
\def\P{{\mathcal {P}}}
\def\GVM{ GVM }
\def\iff{ if and only if  }
\def\add{{\rm add}}
\def\ld{\ldots}
\def\vd{\vdots}
\def\sl{{\rm sl}}
\def\mod{{\rm mod}}
\def\len{{\rm len}}
\def\cd{\cdot}
\def\dd{\ddots}
\def\q{\quad}
\def\qq{\qquad}
\def\ol{\overline}
\def\tl{\tilde}
\def\nn{\nonumber}

\title{On Kostant's theorem for the Lie superalgebra $Q(n)$}

\author{ Elena Poletaeva and Vera Serganova }


\address{ Dept. of Mathematics, University of Texas-Pan American,
Edinburg, TX 78539 } \email{elenap\atSign{}utpa.edu}

\address{ Dept. of Mathematics, University of California at Berkeley,
Berkeley, CA 94720 } \email{serganov\atSign{}math.berkeley.edu}

\maketitle

\section{Introduction}
A finite $W$-algebra is a certain associative algebra attached to a
pair $(\g,e)$ where $\g$ is a complex semisimple Lie algebra and $e\in\g$
is a nilpotent element. Geometrically a finite $W$ algebra is a
quantization of the Poisson structure on the so-called Slodowy slice
(a transversal slice to the orbit of $e$ in the adjoint
representation). In the case when $e=0$ the finite $W$-algebra
coincides with the universal enveloping algebra $U(\g)$ and in the
case when $e$ is a regular nilpotent element, the corresponding
$W$-algebra coincides with the center of $U(\g)$. The latter case was
studied by B. Kostant \cite{Ko} 
who was motivated by applications to generalized Toda lattices. The
general definition of a finite $W$-algebra was given by 
A. Premet in \cite{Pr1}. I. Losev used the machinery of Fedosov
quantization to prove important results relating 
representations of $W$-algebras and primitive ideals of
$U(\g)$ \cite{L1, L2, L3} (see also \cite{Pr2, Pr3, Pr4}). 
He used this result to prove long standing conjectures of A. Joseph
and others concerning primitive ideals in $U(\g)$, \cite{J}.

On the other hand, affine $W$-algebras were first constructed by physicists \cite{FRTW1, FRTW2}.
The role of the Slodowy slice in $W$-algebras in the principal case was recognized in \cite{BFR}.
A. De Sole and V.G. Kac in \cite{DK} established the relation between affine and finite $W$-algebras.

Let us mention an important discovery of physicists, \cite{RS}, that
for $\g=\s\l(n)$ finite $W$-algebras are closely related 
to Yangians. This connection was
further studied in \cite{B} and \cite{BK}.

It is interesting to generalize all above applications to Lie superalgebras.
Finite $W$-algebras for Lie superalgebras
have been extensively studied by
 C. Briot, E. Ragoucy, J. Brundan, J. Brown, S. Goodwin,
 W. Wang, L. Zhao and other mathematicians and physicists \cite{BR, BBG, W, Z}.
Analogues of finite $W$-algebras for Lie superalgebras in terms of BRST cohomology
 were defined in by A. De Sole and V.G. Kac in \cite{DK}.

In \cite{BR} C. Briot and E. Ragoucy observed
that finite $W$-algebras associated with certain nilpotent orbits in
$\g\l(pm|pn)$ can be realized as truncations of 
the super-Yangian of $\g\l(m|n)$, see \cite{M} for definition. 

The principal finite $W$-algebras for $\g\l(m|n)$  associated to regular (principal)
 nilpotent elements were described as certain truncations of a shifted version of the super-Yangian
$ Y(\g\l(1|1))$  in \cite{BBG}. It is also proven there that all irreducible modules over principal
finite $W$-algebras  are  finite-dimensional for
 $\g\l(m|n)$. Furthermore,  \cite{BBG} contains a classification of
 irreducible modules using highest weight theory.

 In \cite{Z} L. Zhao generalized certain results about finite
 $W$-algebras to the case of Lie superalgebras. In particular he
has proved that the definition of a finite $W$-algebra
does not depend on a choice of an isotropic subspace $\l$ and a good $\Z$-grading.
He has also proved an analogue of the Skryabin theorem establishing equivalence between the category
of modules over a finite $W$-algebra and the 
category of generalized Whittaker $\g$-modules. He also gave a
definition of a finite $W$-algebra for the queer Lie superalgebra $Q(n)$ .

In \cite{PS1, P} we described the finite $W$-algebras in the regular
case for some classical and exceptional Lie superalgebras of 
defect one.

In this paper we are interested in the finite $W$-algebra associated
with a regular nilpotent element
$\chi\in \g_{\bar 0}^*$ for  a Lie superalgebra $\g$ with
reductive even part $\g_{\bar 0}$. (Since not all such superalgebras
admit an even invariant form, 
we can not identify $\g$ with $\g^*$, and we use the notation $W_\chi$ instead of $W_e$.)
We prove that for basic classical $\g$ or $Q(n)$ and the regular
$\chi$ the algebra $W_\chi$ satisfies the Amitsur-Levitzki identity (\cite{AL})
(Corrolary \ref{corAL}). In the proof we use some sort of reduction by
constructing an injective homomorphism $\vartheta:W_\chi\to \bar W_{\chi}^\s$, where
$\s$ is the reductive part of some parabolic subalgebra  $\p\subset\g$,
and $\bar W_\chi^\s$ is an analogue of $W_\chi$ for $\s$. 
As a corollary we obtain that all irreducible representations of
$W_{\chi}$ are 
finite-dimensional (Proposition \ref{irred}).

We study in detail the case when $\g = Q(n)$ and $\chi$ is regular.
In this case, $\p$ is a Borel subalgebra and $\s$ is a Cartan subalgebra.
We obtain results about the image of $\vartheta$ in this case, which imply, in particular, that the center of
$W_{\chi}$ coincides with the center of $U(Q(n))$ (Corollary \ref{cent2}).

Using Sergeev's construction of certain elements in the universal enveloping algebra $U(Q(n))$ ([24]),
we construct generators of $W_{\chi}$. Using these generators, we prove that the associated
graded algebra $Gr_K W_{\chi}$ with respect to the Kazhdan filtration is isomorphic to $S(\g^{\chi})$
(the symmetric algebra of the annihilator $\g^{\chi}$ of $\chi$ in $\g$)
(Conjecture \ref{generalconjecture} and Corollary \ref{cor}).
Furthermore, we prove that $W_{\chi}$ is isomorphic to a quotient of the super-Yangian of $Q(1)$
defined by M. Nazarov and A. Sergeev (\cite{N, NS}) (Theorem \ref{Yan}).
 Finally, we construct $n$ even and $n$ odd generators in
$W_{\chi}$, such that all even generators commute and generate the polynomial subalgebra of rank $n$ in
$W_{\chi}$, and the commutators of odd generators lie in the center of $W_{\chi}$ (Theorem \ref{Gen}).

{\bf Acknowledgments.} The authors would like to thank J. Brundan, M. Nazarov and A. Sergeev
for helpful discussions. A significant part of this work was done at the Max-Planck-Institut fur Mathematik in Bonn
in the Fall of 2012. We thank the MPIM for the hospitality and support. The second author was also supported by NSF grant
DMS - 1303301.

\section{Finite $W$-algebras for Lie superalgebras}

\subsection{Definitions}
Let $\g = \g_{\bar 0}\oplus \g_{\bar 1}$ be a Lie superalgebra with reductive even part $\g_{\bar 0}$.
Let $\chi\in \g_{\bar 0}^*\subset\g^*$ be an even nilpotent element
in the coadjoint representation.
\footnote{ Denote by $G_{\bar 0}$ the algebraic reductive group of $\g_{\bar 0}$. Then $\chi$ is nilpotent if the closure of
$G_{\bar 0}$-orbit in $\g_{\bar 0}^*$ contains zero.}
By $\g^\chi$ we denote the annihilator of $\chi$ in $\g$. By definition
$$\g^\chi=\{x\in\g \hbox{ }|\hbox{ }\chi([x,\g])=0\}.$$
A good $\mathbb Z$-grading for $\chi$ is a  $\mathbb Z$-grading
$\displaystyle\g=\bigoplus_{j\in\Z} \g_j $ satisfying the following two conditions
\begin{enumerate}
\item $\chi(\g_j)=0$ if $j\neq -2$;
\item $\g^\chi$  belongs to $\displaystyle\bigoplus_{j\geq 0} \g_j$.
\end{enumerate}

Note that $\chi([\cdot,\cdot]):\g_{-1}\times\g_{-1}\to\mathbb C$ is a non-degenerate skew-symmetric even bilinear form on
$\g_{-1}$. Let $\l$ be a maximal isotropic subspace with respect to
this form. We consider
a nilpotent subalgebra
$\m = (\displaystyle\bigoplus_{j\leq -2}\g_j)\oplus \l$ of $\g$.
The restriction of $\chi$ to $\m$
$$\chi: \m\longrightarrow \C$$
defines a one-dimensional representation
$C_{\chi} = <v>$ of $\m$.

\begin{definition} The induced $\g$-module
$$Q_{\chi} := U(\g)\otimes_{U(\m)}C_{\chi} \cong U(\g)/I_{\chi},$$
where $I_{\chi}$ is the left ideal of $U(\g)$
generated by $a - \chi(a)$ for all $a\in \m$,
is called {\it the generalized Whittaker module}.
\end{definition}

\begin{definition} \hbox{ }\cite{Pr1}.
Define {\it the finite $W$-algebra} associated to the nilpotent element
$\chi$ to be
$$W_{\chi} := \hbox{End}_{U(\g)}(Q_{\chi})^{op}.$$
\end{definition}

As in the Lie algebra case, the superalgebras $W_\chi$ are all
isomorphic for different choices of good gradings and maximal
isotropic subspaces $\l$ \cite{Z}.

If $\g$ admits an even non-degenerate invariant supersymmetric bilinear form, then $\g\simeq \g^*$
and $\chi(x)=(e|x)$ for some nilpotent $e\in \g_{\bar 0}$.
By the Jacobson--Morozov theorem  $e$ can be included in $\s\l(2) = <e, h, f>$.
As in the Lie algebra case, the linear operator $\hbox{ad} h$ defines a Dynkin $\Z$-grading
$\displaystyle \g = \bigoplus_{j\in \Z}\g_j$, where
$$\g_j = \lbrace x\in \g \hbox{ }|\hbox{ }\hbox{ad} h(x) = jx\rbrace.$$
As follows from representation theory of $\s\l(2)$, the Dynkin $\Z$-grading
is good.
Let $\g^e := \hbox{Ker}(\hbox{ad} e)$. Note that as in the Lie algebra case,
$\dim \g^e = \dim \g_0 + \dim \g_1$ and
$\displaystyle \g^e \subseteq \bigoplus_{j\geq 0}\g_j$.

Most results of this paper concern the case when $\g$ admits an odd
non-degenerate invariant supersymmetric bilinear form. In this case
$\g\simeq \Pi\g^*$ and $\chi(x)=(E|x)$ for some nilpotent $E\in \g_{\bar 1}$.
Among classical Lie superalgebras only $Q(n)$  or $PSQ(n)$ admit
an odd non-degenerate invariant supersymmetric bilinear form. We will
see that in this case there is an analogue of the Dynkin $\Z$-grading.

Note that by Frobenius reciprocity
$$\hbox{End}_{U(\g)}(Q_{\chi})= \hbox{Hom}_{U(\m)}(C_{\chi},Q_{\chi}).$$
That defines an identification of $W_\chi$ with the subspace
$$Q_\chi^\m=\{u\in Q_{\chi}\hbox{ }| \hbox{ }au=\chi(a)u\hbox{ for all } a\in \m\}.$$
In what follows we denote by  $\pi:U(\g)\to U(\g)/I_\chi$  the natural
projection. By above
\begin{equation}\label{favorite}
W_{\chi} = \lbrace \pi(y) \in U(\g)/I_{\chi} \hbox{ }| \hbox{ } (a - \chi(a))y\in I_{\chi}\hbox{ for all } a\in \m\rbrace,
\end{equation}
or, equivalently,
\begin{equation}\label{favorite1}
W_{\chi} = \lbrace \pi(y) \in U(\g)/I_{\chi} \hbox{ }| \hbox{ ad} (a)y\in I_{\chi}\hbox{ for all } a\in \m\rbrace.
\end{equation}
The algebra structure on $W_{\chi}$ is given by
$$\pi(y_1)\pi({y}_2) = \pi({y_1y_2})$$
\vskip 0.1in
\noindent
for $y_i\in U(\g)$ such that $\hbox{ad}(a)y_i\in I_{\chi} \hbox{ for all }  a\in \m$
and $i = 1, 2$.

\begin{definition}
A $\Z$-grading $\displaystyle\g = \bigoplus_{j\in \Z}\g_j$ is called {\it even},
if $\g_j = 0$ unless $j$ is an even integer.
\end{definition}

The definition of $W_{\chi}$ for an even good $\Z$-grading is simpler, since
in this case
$\g_{-1} = 0$. Hence there is no complications of choice of a Lagrangian subspace $\l$
and $\displaystyle\m = \bigoplus_{j\geq 1}\g_{-2j}$.

Let $\displaystyle\p : = \bigoplus_{j\geq 0}\g_{2j}$. It follows directly from
definition that $\p$ is a parabolic subalgebra of $\g$.
From the PBW theorem,
$$U(\g) = U(\p) \oplus I_{\chi}.$$
The projection $pr: U(\g) \longrightarrow U(\p)$
along this direct sum decomposition induces an isomorphism:
$U(\g)/I_{\chi} \buildrel \sim \over \longrightarrow U(\p)$. Thus,
the algebra $W_{\chi}$ can be regarded as a {\it subalgebra} of $U(\p)$.

\subsection{Kazhdan filtration on $W_{\chi}$} Define the $\Z$-grading
on $T(\g)$ induced by the shift by $2$ of the fixed good $\Z$-grading. In
other words, we set the degree of $X\in \g_j$ to be $j+2$. It induces a
filtration on $U(\g)$ and therefore on $U(\g)/I_\chi$, which is called
the {\it Kazhdan filtration}. We
will denote  by $\hbox{Gr}_K$ the corresponding graded algebras. Recall that by (\ref{favorite})
$W_\chi\subset U(\g)/I_\chi$. Hence we have the induced filtration on $W_\chi$.
It is not hard to see that $Gr_KU(\g)$ is supercommutative and
therefore  $\hbox{Gr}_K W_\chi$ is also supercommutative.  For any $X\in W_\chi$ we denote by $\hbox{Gr}_K X$ the corresponding element in
$\hbox{Gr}_K W_\chi$.
The following result is very important.

\begin{theorem}\label{premet} A. Premet \cite{Pr1}. Let $\g$ be a semisimple Lie algebra.
Then the associated graded algebra $Gr_KW_{\chi}$ is isomorphic to $S(\g^\chi)$.
\end{theorem}

We believe that the above theorem holds for basic classical Lie
superalgebras if  $\dim(\g_{-1})_{\bar 1}$ is {\it even}.
 In fact, for $\g=\mathfrak{gl}(m|n)$ and
regular $\chi$ it is proven in \cite{BBG}.
In this paper we prove the
analogous result for regular $\chi$ and $\g=Q(n)$ (see Corollary \ref{cor}).

We will prove now a weaker general result. Let $\l'$ be some subspace
in $\g_{-1}$ 
satisfying the following two properties
\begin{itemize}
\item $\g_{-1}=\l\oplus\l'$;
\item $\l'$ contains a maximal isotropic subspace with respect to the form $\chi([\cdot,\cdot])$ on $\g_{-1}$.
\end{itemize}

If $\hbox{dim}(\g_{-1})_{\bar 1}$ is even, then $\l'$ is a maximal
isotropic subspace. 
If $\hbox{dim}(\g_{-1})_{\bar 1}$ is
odd, then $\l^\perp\cap \l'$ is one-dimensional and we fix
$\theta\in \l^\perp\cap \l'$ such that 
$\chi([\theta,\theta])=2$. It is clear that $\pi(\theta)\in W_\chi$
and $\pi(\theta)^2=1$.

Let $\displaystyle\p=\bigoplus _{j\geq 0}\g_j$.
By the PBW theorem, $U(\g)/I_\chi\simeq S(\p\oplus\l')$ as a vector space. Therefore
$\hbox{Gr}_K(U(\g)/I_\chi)$ is isomorphic to $S(\p\oplus\l')$ as a vector space.
The good grading of $\g$ induces the grading on $S(\p\oplus\l')$.
For any $X\in S(\p\oplus\l')$ we denote by $\bar X$ the element of highest degree in this grading.
Following the original Premet's proof we will prove now the following statement.

\begin{theorem}\label{PremetQ}

(a) Assume that $\hbox{dim}(\g_{-1})_{\bar 1}$ is even. If $X\in\hbox{Gr}_KW_\chi$, then $\bar X\in S(\g^\chi)$.

\noindent
(b) Assume that $\hbox{dim}(\g_{-1})_{\bar 1}$ is odd. If
$X\in\hbox{Gr}_KW_\chi$, then
$\bar X\in S(\g^\chi\oplus\C\theta)$.

\end{theorem}
\begin{proof}
We start with the following simple observation.
\begin{lemma}\label{aux} Let $x\in\p\oplus\l'$. Then
$\chi([\m,x])=0$ if and only if  $x\in\g^\chi$ for even  $\hbox{dim}(\g_{-1})_{\bar 1}$ and $x\in\g^\chi\oplus\C\theta$ for odd  $\hbox{dim}(\g_{-1})_{\bar 1}$.
\end{lemma}
\begin{proof} Note that if $x\in\g_i$ and $Y\in\g_j$, then $\chi([Y,x])\neq 0$ implies $i+j=-2$. Therefore if $x\in\p$, the condition
$\chi([\m,x])=0$ implies the condition $\chi([\g,x])=0$, and thus $x\in \g^\chi$. If $x\in\l'$, then the condition $\chi([\m,x])=0$ is equivalent to the condition
$\chi([\l,x])=0$. Therefore $x\in \l^\perp\cap \l'=\mathbb C\theta$.
\end{proof}
Let $X\in Gr_KW_{\chi}$. Passing to the graded version of (\ref{favorite1}) we obtain that for any $Y\in\m$ we have
\begin{equation}\label{equP}
\pi(\hbox{ad}Y(X))=0.
\end{equation}
Define $\gamma:\m\otimes S(\p\oplus \l')\to S(\p\oplus \l')$ by putting
$$\gamma(Y,Z)=\pi(\hbox{ad}Y(Z))$$
for all $Y\in\m,Z\in S(\p\oplus\l')$. It is easy to see that if
$Y\in \g_{-i}$, 
where $i > 0$, and $Z\in  S(\p\oplus\l')_j$, then
$\gamma(Y,Z)\in S(\p\oplus\l')_{j-i}\oplus  S(\p\oplus \l')_{j-i+2})$. Hence we can
write $\gamma=\gamma_0+\gamma_2$ where $\gamma_0(Y,Z)$ is the projection on $S(\p\oplus\l')_{j-i}$ and $\gamma_2(Y,Z)$ is the projection on $S(\p\oplus\l')_{j-i+2}$.
The condition (\ref{equP}) implies that for any  $X\in Gr_KW_{\chi}$
\begin{equation}\label{equP1}
\gamma_2(\m,\bar X)=0.
\end{equation}
On the other hand, $\gamma_2:\m\times S(\p\oplus\l')\to S(\p\oplus\l')$ is a derivation with respect to the second argument defined by the condition
$$\gamma_2(Y,Z)=\chi([Y,Z])$$
for any $Y\in\m,Z\in\p\oplus\l'$. Now by induction on the polynomial degree of $\bar X$ in $S(\p\oplus\l')$, using Lemma \ref{aux},
one can show that  (\ref{equP1}) implies $\bar X\in S(\g^\chi)$ (respectively, $\bar X\in S(\g^\chi\oplus\C\theta)$).
\end{proof}

\begin{proposition}\label{towpr} Assume that  $\hbox{dim}(\g_{-1})_{\bar 1}$ is even (respectively, odd). Let $y_1,\dots,y_p$ be a basis in $\g^\chi$ homogeneous in the good $\Z$-grading.
Assume that there exist $Y_1,\dots, Y_p\in W_\chi$ such that $\overline{\hbox{Gr}_K Y_i}=y_i$ for all $i=1,\dots, p$.

(a) $Y_1,\dots,Y_p$ generate $W_\chi$ (respectively,  $Y_1,\dots,Y_p$ and $\pi(\theta)$ generate $W_\chi$) ;

(b) $\hbox{Gr}_KW_\chi\simeq S(\g^\chi)$ (respectively,
$\hbox{Gr}_KW_\chi\simeq S(\g^\chi)\otimes \C[\xi]$, where 
$\C[\xi]$ is the exterior algebra generated by one element $\xi$).
\end{proposition}
\begin{proof} We will give a proof in the case  when $\hbox{dim}(\g_{-1})_{\bar 1}$ is even. The odd case is analogous and we leave it to the reader.
Let us first prove (a) by contradiction. Assume that $X\in W_\chi$ is an element of minimal Kazhdan degree such that it does not lie in the
subalgebra generated by $Y_1,\dots,Y_p$. By Theorem \ref{PremetQ} we have
$$\overline  {\hbox{Gr}_K X}=\sum c(a_1,\dots,a_p)y_1^{a_1}\dots y_p^{a_p}.$$
Let
$$Z=X-\sum c(a_1,\dots,a_p)Y_1^{a_1}\dots Y_p^{a_p}.$$
Then Kazhdan degree of $Z$ is less than that of $X$. By minimality of degree of $X$ we conclude that $Z=0$. That contradicts our assumption.

To prove (b) write $\p=\g^\chi\oplus \mathfrak r$, where $\mathfrak r$ is some graded subspace complementary to $\g^\chi$. Let
$\gamma: S(\p\oplus \l')\to S(\g^\chi)$ denote the natural projection with kernel $(\mathfrak r\oplus \l') S(\p\oplus \l')$. By (a) and Theorem \ref{PremetQ}
the restriction
$\gamma: \hbox{Gr}_KW_\chi\to  S(\g^\chi)$ is an isomorphism of rings.
\end{proof}

\begin{conjecture}\label{generalconjecture}
 Assume that $\g$ is a Lie superalgebra with reductive even part $\g_{\bar 0}$.
 If $\hbox{dim}(\g_{-1})_{\bar 1}$ is even, then $Gr_KW_\chi\simeq S(\g^\chi)$ and if
$\hbox{dim}(\g_{-1})_{\bar 1}$ is odd, then  $Gr_KW_\chi\simeq S(\g^\chi)\otimes \C[\xi]$, where $\C[\xi]$ is the exterior algebra generated by one element $\xi$.
\end{conjecture}

\subsection{ Kostant's theorem and the regular case for Lie superalgebras}
A nilpotent $\chi\in\g^*_{\bar 0}$ is called {\it regular} if $G_{\bar 0}$-orbit of $\chi$ has
maximal dimension, i.e. the dimension of $\g^\chi_{\bar 0}$ is minimal.
Let us recall that for a regular nilpotent $\chi$ and
a reductive Lie algebra $\g$ the algebra $W_\chi$ is isomorphic to
the center $Z(\g)$ of $U(\g)$, see \cite{Ko}.

It is not hard to see that this result of B. Kostant does not hold for Lie superalgebras.
In Section 3 we will prove that for regular $\chi$, $W_\chi$ satisfies
the
Amitsur--Levitzki identity and all irreducible representations of $W_\chi$
are finite-dimensional with dimension not greater than $2^{k+1}$,
where $k$ is the constant depending on defect of $\g$ and the parity
of $\hbox{dim}\g^\chi_{\bar 1}$.
Recall that for a contragredient $\g$ the defect of $\g$ is the
maximal number of mutually orthogonal linearly independent isotropic
roots, \cite {KW}.

\subsection{Good $\Z$-gradings for superalgebras in the regular case} Good $\Z$-gradings for basic
classical superalgebras are classified in \cite{H}. In the case when $\chi$ is
regular and $\g$ is of type II (i.e. $\g_{\bar 0}$ is semisimple and
$\g_{\bar 1}$ is a simple  
$\g_{\bar 0}$-module), the only good
$\Z$-grading is the Dynkin $\Z$-grading, and it is never even. If $\g$ is of
type I, i.e. $\g_{\bar 0}$ has a non-trivial center, we can
choose an even good $\Z$-grading for any $\chi$. For the Lie superalgebra $Q(n)$ the
analogue of Dynkin $\Z$-grading is even for any $\chi$.

Let us concentrate on the case of basic classical or exceptional Lie
superalgebras of type II and regular $\chi$. 
In this case
$\chi(\cdot)=(e|\cdot)$ for some principal nilpotent element
$e\in\g_{\bar 0}$.
We are going to describe the Dynkin $\Z$-grading on $\g$ in terms of a specific Borel subalgebra.
Let $\b_{\bar 0}\subset\g_{\bar 0}$ be the Borel subalgebra containing $e$. Since $e$ is principal, this Borel subalgebra is unique.
Let $\Pi_0$ denote the set of simple roots of $\b_{\bar 0}$.

\begin{lemma}\label{specialborel} Let $\g$ be a basic classical or exceptional Lie superalgebra of type II.

(a) There exists a Borel subalgebra $\b_{\bar 0}\subset\b\subset\g$ with the set of simple roots $\Pi$ such that for any
root $\beta\in\Pi_0$ either $\beta\in \Pi$ or $\beta=\alpha_1+\alpha_2$ for some $\alpha_1,\alpha_2\in\Pi$.

(b) Let $d$ denote the defect of $\g$. Then the number of odd roots in
$\Pi$ equals $2d$ if 
$\g=\mathfrak {osp}(2m+1|2n)$ for $m\geq n$,
$\mathfrak {osp}(2m|2n)$ for $m\leq n$ or $G_3$, and  the number of
odd roots in $\Pi$ equals $2d+1$ if 
$\g=\mathfrak {osp}(2m+1|2n)$ for $m< n$,
$\mathfrak {osp}(2m|2n)$ for $m> n$, $D(2,1;a)$ or $F_4$.

(c) Let $e,h,f$ be the $\mathfrak{sl}(2)$-triple such that
$h\in\h$. Then $\alpha(h)=2$ for any even $\alpha\in \Pi$ and
$\alpha(h)=1$ for any odd $\alpha\in\Pi$, i.e. the Dynkin $\Z$-grading is consistent.
\end{lemma}
\begin{proof} (a) Among all Borel subalgebras containing $\b_{\bar 0}$
pick up the one that has maximal number of 
odd roots and contains an odd non-isotropic
root if such roots exist. For ortho-symplectic superalgebra those Borel subalgebras are listed in \cite{GS}.

For the exceptional superalgebras we list the simple roots using the roots
description in \cite{K}. If $\g=G_3$, the set of simple roots is
$\{\delta, \gamma_1-\delta,\gamma_2\}$, where $\gamma_1$ is the short
and $\gamma_2$ is the long simple root of $G_2$.
If $\g=F_4$, then the set of simple roots
is 
$$\{\varepsilon_1-\varepsilon_2,
\frac{1}{2}(-\varepsilon_1+\varepsilon_2-\varepsilon_3-\delta), 
\frac{1}{2}(-\varepsilon_1-\varepsilon_2+\varepsilon_3+\delta),
\frac{1}{2}(\varepsilon_1+\varepsilon_2-\varepsilon_3+\delta)\}.$$

(b) follows by direct inspection.

(c) follows from the condition $[h,e]=2e$ and (a).
\end{proof}

\begin{corollary}\label{defect}  Let $\g$ be a basic classical or
exceptional Lie superalgebra of type II, and 
$d$ be its defect.
If $\g=\mathfrak {osp}(2m+1|2n)$ for $m\geq n$,
$\mathfrak {osp}(2m|2n)$ for $m\leq n$ or $G_3$, then
$\hbox{dim}{}\g_{-1}=(0|2d)$. 
If $\g=\mathfrak {osp}(2m+1|2n)$ for $m< n$,
$\mathfrak {osp}(2m|2n)$ for $m> n$, $D(2,1;a)$ or $F_4$, then
$\hbox{dim}{}\g_{-1}=(0|2d+1)$. By Lemma \ref{specialborel}(c)
$\hbox{dim}\g_{-1}$ equals the number of irreducible
$\mathfrak{sl}(2)$-components in $\g_{\bar 1}$. Therefore   $\dim(\g^{\chi})_{\bar 1}  = 2d$ or $2d+1$.
\end{corollary}

\begin{corollary}\label{specialparabolic} Let $\g$ satisfy the assumptions of Corollary \ref{defect}.

(a) One can choose a maximal isotropic subspace $\l\subset\g_{-1}$
such that 
$\l=\g_{-\alpha_1}\oplus\dots\oplus\g_{-\alpha_d}$ for some
isotropic mutually orthogonal roots
$\alpha_1,\dots,\alpha_d\in\Pi$. In particular, $[\l,\l]$=0.

(b)  There exists a parabolic subalgebra $\p\subset\g$ with Levi
subalgebra $\s$ such that 
$\m\cap \s$ is an even one dimensional subspace, and if
$\n^-$ denotes the nil radical of the opposite parabolic $\p^-$, then $\n^-\subset\m$.

(c) If $\g$ does not have non-isotropic roots
(i.e. $\g=\mathfrak{osp}(2m|2n)$ or $F_4$), 
then $[\s,\s]$ is isomorphic to a direct sum of several copies
of $\mathfrak{sl}(1|1)$ and one copy
of $\mathfrak{sl}(1|2)$. If $\g$ has non-isotropic roots (i.e. $\g=\mathfrak{osp}(2m+1|2n)$ or $G_3$),
then $[\s,\s]$ is isomorphic to a direct sum of several copies of $\mathfrak{sl}(1|1)$ and one copy of
$\mathfrak{osp}(1|2)$.

\end{corollary}
\begin{proof} Let $\Gamma$ denote the Dynkin diagram of $\Pi$. For any
subset $C\subset \Pi$ we denote by $\Gamma_C$ 
the corresponding subdiagram of $\Gamma$.
Let $\Pi'$ denote the set of all odd roots of $\Pi$, the subgraph
$\Gamma_{\Pi'}$ is connected and  $\Pi'$ has at most 
one non-isotropic root.
Let us choose a subset $A=\{\alpha_1,\dots\alpha_d\}\subset \Pi'$ of
mutually orthogonal isotropic roots such that the 
subgraph $\Gamma_{\Pi'\setminus A}$ has
maximal number of connected components.
If $\Pi'$ contains a non-isotropic root, then $\Gamma_{\Pi'\setminus A}$ is a disjoint union of single vertex diagrams. If all roots of $\Pi'$ are isotropic, then
$\Gamma_{\Pi'\setminus A}$ is a disjoint union of several single
vertex diagrams and one diagram consisting of 
two connected isotropic vertices.
The latter is the diagram of  $\mathfrak{sl}(1|2)$.

Now we set $\s$ to be the subalgebra of $\g$ generated by $\h$ and $\g_{\pm \beta}$ for all $\beta\in \Pi'\setminus A$
and let $\p=\b+\s$. We leave to the reader to check that all requirements
of the corollary are true for this choice.
\end{proof}

\begin{example} Let $\g=\mathfrak{osp}(3|4)$. Then $\Pi$ has the Dynkin diagram
$$\otimes-\otimes\Rightarrow\bullet,$$
and $A$ consists of one middle vertex. In this case $[\s,\s]\simeq \mathfrak{sl}(1|1)\oplus\mathfrak{osp}(1|2)$.
\end{example}
\begin{example}
Let $\g=G_3$. Then $\Pi$ has the Dynkin diagram 
$$\bullet\Leftarrow\otimes\Lleftarrow \circ,$$
and  $A$ again coincides with the midle vertex. In this case we also
have $[\s,\s]\simeq \mathfrak{sl}(1|1)\oplus\mathfrak{osp}(1|2)$.
\end{example}

\subsection{The queer superalgebra $Q(n)$} Recall that
the {\it queer} Lie superalgebra is defined as follows
$$Q(n) := \lbrace
\left(\begin{array}{c|c}
A&B\\
\hline
B&A\\
\end{array}\right)\hbox{ }|\hbox{ }A, B \hbox{ are } n\times n \hbox{ matrices}\rbrace.$$
Let $\hbox{otr} \left(\begin{array}{c|c}
A&B\\
\hline
B&A\\
\end{array}\right) = \hbox{tr} B$.

\noindent
\begin{remark}
$Q(n)$ has one-dimensional center $<z>$, where $z = 1_{2n}$.
Let
$$S{Q}(n) = \lbrace X\in Q(n)
\hbox{ }|\hbox{ otr}X = 0\rbrace.$$

\noindent
The Lie superalgebra $\tilde{Q}(n) := S{Q}(n)/<z>$ is simple for $n\geq 3$, see \cite{K}.
\end{remark}

Note that
$\g = Q(n)$ admits an {\it odd} non-degenerate $\g$-invariant supersymmetric bilinear form
$$(x | y) := \hbox{otr}(xy) \hbox{ for } x, y\in \g.$$
Therefore,
we identify the coadjoint module $\g^*$ with $\Pi(\g)$, where $\Pi$ is
the  change of parity functor.

\noindent
Let $e_{i,j}$ and $f_{i,j}$ be standard bases in $\g_{\bar 0}$ and $\g_{\bar 1}$ respectively:

\vskip 0.2in
$$e_{i,j} =
\left(\begin{array}{c|c}
E_{ij}&0\\
\hline
0&E_{ij}\\
\end{array}\right), \quad
f_{i,j} =
\left(\begin{array}{c|c}
0&E_{ij}\\
\hline
E_{ij}&0\\
\end{array}\right),$$
where $E_{ij}$ are elementary $n\times n$  matrices.

Let $\s\l(2) = <e, h, f>$, where
$$e = \sum_{i=1}^{n-1}e_{i,i+1},
\quad h = \hbox{diag}(n-1, n-3, \ldots, 3-n, 1-n), \quad
f = \sum_{i=1}^{n-1}i(n-i)e_{i+1,i}.$$
Note that $e$ is a regular nilpotent element,
$h$ defines an even Dynkin $\Z$-grading of
$\g$ whose degrees on the elementary matrices are

\begin{eqnarray}
&\left(\begin{array}{cccc|cccc}
0&2&\cdots&2n-2&0&2&\cdots&2n-2\nonumber\\
-2&0&\cdots&2n-4&-2&0&\cdots&2n-4\\
\cdots&\cdots&\cdots&\cdots&\cdots&\cdots&\cdots&\cdots\\
2-2n&\cdots&\cdots&0&2-2n&\cdots&\cdots&0\\
\hline
0&2&\cdots&2n-2&0&2&\cdots&2n-2\nonumber\\
-2&0&\cdots&2n-4&-2&0&\cdots&2n-4\\
\cdots&\cdots&\cdots&\cdots&\cdots&\cdots&\cdots&\cdots\\
2-2n&\cdots&\cdots&0&2-2n&\cdots&\cdots&0\\
\end{array}\right).
\end{eqnarray}

\noindent
Let  $E = \sum_{i=1}^{n-1}f_{i,i+1}$. Since we have an isomorphism $\g^*\simeq\Pi(\g)$, an even regular nilpotent $\chi\in\g^*$ can be defined by
$\chi(x) := (x|E)$ for $x \in \g$.
Note that the Dynkin $\Z$-grading is good for $\chi$.
We have that
\begin{equation}\label{equ1}
\g^\chi=\g^E = \lbrace z, e, e^2, \ldots, e^{n-1} \hbox{ }|\hbox{ } H_0, H_1, \ldots, H_{n-1} \rbrace,\quad  \hbox{dim}(\g^E) = (n|n),
\end{equation}
where
$H_k = \sum_{i=1}^{n-k}(-1)^{i+k-1}f_{i,i+k}$ for $k = 0, \ldots, n-1$.
Let
$$\m = \bigoplus_{j = 1}^{n-1}\g_{- 2j}.$$
Note that $\m$ is generated by $e_{i+1, i}$ and $f_{i+1, i}$, where $i = 1, \ldots, n-1$, and
\begin{equation}\label{equ2}
\chi(e_{i+1, i}) = 1, \quad \chi(e_{i+k, i}) = 0 \hbox{ if  } k\geq 2,\quad \chi(f_{i+k, i}) = 0
\hbox{ if  } k\geq 1.
\end{equation}
The left ideal $I_\chi$ and $W_\chi$ are defined now as usual. Moreover,
$$\b:=\bigoplus_{j = 0}^{n-1}\g_{2j}$$
is a Borel subalgebra of $\g$, $\h:=\g_0$ is a Cartan subalgebra, and $\b=\h\oplus\n$, where
$$\n:=\bigoplus_{j = 1}^{n-1}\g_{2j}.$$
Note that the algebra $W_\chi$ can be regarded as a {\it subalgebra} of $U(\b)$.

\section{Some general results}
\subsection{The Harish-Chandra homomorphism}
In this section we assume that $\g$ is a basic classical Lie
superalgebra or $Q(n)$. Let $\p\subset \g$ be a 
parabolic subalgebra
such that $\n^-\subset\m\subset \p^-$, where $\n^-$ denotes the nilradical of the opposite parabolic $\p^-$.
Let $\s$ be the Levi subalgebra of $\p$, $\n$ be its nilradical and
$\m^{\s}=\m\cap \s$. Note that $\m=\n^-\oplus \m^s$. We denote by 
$Q_\chi^\s$ the induced module
$U(\s)\otimes_{U(\m^\s)}C_\chi$, where by $\chi$ we
understand the restriction of $\chi$ on $\s$. Let 
$$\bar W_\chi^\s=\hbox{End}_{U(\s)}Q_\chi^\s=(Q_\chi^\s)^{\m^\s}.$$ 

Let $J_\chi$ (respectively $J_\chi^\s$)
be the left ideal in $U(\p)$ (respectively in $U(\s)$) generated
by $a-\chi(a)$ for all $a\in \m^\s$.  

Finally, let
$\bar\vartheta:U(\p)\to U(\s)$ denote the projection with the kernel
$\n U(\p)$. Note that 
$\bar \vartheta(J_\chi)=J_\chi^\s$. Thus,
the projection $\vartheta':U(\p)/J_\chi\to U(\s)/J_\chi^\s$ is well defined.

Note that we have an isomorphism of vector spaces $Q_\chi\simeq U(\p)/J_\chi$, hence $W_\chi$ 
can be identified with  
a subspace in $(U(\p)/J_{\chi})^{\m^\s}$. On the other hand, $\bar W_\chi^\s$
can be identified with the subspace $(U(\s)/J_{\chi}^s)^{\m^\s}$. 
Consider a map $\vartheta: W_\chi\to U(\s)/J^\s_\chi$ obtained by the restriction of
$\vartheta'$ to $W_\chi$. Since $\hbox{ad}{}\m^\s(\n)\subset\n$,
$\vartheta$ maps $\hbox{ad}{}\m^\s$-invariants to 
$\hbox{ad}{}\m^\s$-invariants. In other words,
$\vartheta(W_\chi)\subset \bar W_\chi^\s$. Furthermore, one can easily see
that $\vartheta: W_\chi\to \bar W_\chi^\s$ is a 
homomorphism of algebras.

An important example is as follows. Assume that $\g$ admits an even good grading with respect to $\chi$. Then we can set
$\p=\displaystyle \bigoplus_{i\geq 0}\g_i$. Then $\s=\g_0$, $\m^\s=0$
and $\vartheta$ is a homomorphism 
$W_\chi\to U(\s)$.

\begin{theorem}\label{HCinjective}
The homomorphism $\vartheta: W_\chi\to \bar W_\chi^\s$ is injective.
\end{theorem}
\begin{proof} We consider a new grading $\g=\displaystyle \bigoplus_{i\in\mathbb Z}\g_{(i)}$ such that $\p=\displaystyle \bigoplus_{i\geq 0}\g_{(i)}$.
Note that $J_\chi$ is a graded ideal and hence $Q_\chi$ is also a
graded vector space with respect to this new grading. Note that $(Q_\chi)_{(0)}=Q_\chi^\s$.
For any $t\in\mathbb C\setminus\{0\}$ let $\phi_t$ denote the
automorphism of $\g$ that multiplies elements 
of $\g_{(j)}$ by $t^j$.
Let $X\in W_\chi=Q_\chi^\m$. Write $$X=\sum_{i=d}^s X_{(i)},$$
where $X_{(i)}\in (Q_\chi)_{(i)}$ and $X_{(d)}\neq 0$. Our goal is to show that $d=0$.
Let
$$\chi_0=\lim_{t\to 0}{\phi_t(\chi)}.$$
Then $\chi_0(\n^-)=0$, $\chi_0|_{\m^\s}=\chi|_{\m^\s}$.
Note that
$$t^{-d}\phi_t(X)\in W_{\phi_t(\chi)}$$
and hence by the standard continuity argument $X_{(d)}\in (U(\g)\otimes_{U(\m)}C_{\chi_0})^\m$.
Note that 
$$U(\g)\otimes_{U(\m)}C_{\chi_0}=U(\g)\otimes_{U(\p^-)}U(\p^-)\otimes_{U(\m)}C_{\chi_0}.$$
Furthermore, $U(\p^-)\otimes_{U(\m)}C_{\chi_0}$ has the trivial action
of $\n^-$ and is isomorphic to $Q_\chi^\s$ as an $\s$-module.
Thus, $X_{(d)} \in (U(\g)\otimes_{U(\p^{-})}Q_{\chi}^{\s})^{\m}$.

We need now the following Lemma.
\begin{lemma}\label{parind}
$(U(\g)\otimes_{U(\p^{-})}Q_{\chi}^{\s})^{\n^-}=Q_\chi^\s$. 
\end{lemma}
\begin{proof} Let $\zeta$ be a generic central character of $U(\s)$ and $S$ be a quotient of  $Q_\chi^\s$
admitting this central character. Consider the parabolically induced module $M:=U(\g)\otimes_{U(\p^-)}S$
(here we assume that $\n^-$ acts trivially on $S$). We will prove that $M^{\n^-}=S$.

 Let $\gamma:Z(\g)\to Z(\s)$ be the restriction of the Harish-Chandra
 projection 
$U(\g)\to U(\s)$ with kernel $\n^-U(\g)+U(\g)\n^+$.
Note that $M$ admits central character $\gamma^*(\zeta)$. Any simple $\s$-submodule $N\subset M^{\n^-}$ that
admits central character $\zeta'$ generates in $M$ a submodule admitting central character $\gamma^*(\zeta')$.
Hence we have $\gamma^*(\zeta)=\gamma^*(\zeta')$.

Recall the correspondence between central characters and weights. One
chooses a Borel subalgebra $\b^\s$ in $\s$ 
and set $\zeta_\lambda$ to be the central
character of
the Verma module over $\s$ with highest weight $\lambda$. Furthermore
$\b^\s\oplus\n^-$ is a Borel subalgebra in 
$\g$ and we define the $U(\g)$-central character
$\bar{\zeta}_\lambda$ to be the central character of the Verma module
over $\g$ with highest weight $\lambda$. 
Obviously, $\gamma^*(\zeta_\lambda)=\bar{\zeta}_\lambda$.
Moreover, all simple $\s$-subquotients of $M$ admit central character
$\zeta_\mu$ for some 
$\mu\in \lambda+R(\n^-)$ where $R(\n^-)$ is the set of weights of
$U(\n^-)$.
Recall that if $\lambda$ is typical, then
$\bar{\zeta}_\lambda=\bar{\zeta}_\nu$ implies that 
$\nu$ is obtained from $\lambda$ by the shifted action of the Weyl
group of $\g_0$.
 
Let us choose a typical $\lambda$ such that the intersection of the
orbit of $\lambda$ and $\lambda+R(\n^-)$  
equals $\lambda$.
Suppose that there exists a simple $N\subset M^{\n^-}\cap \n^-
M$. Then $N$ admits $U(\s)$-central character 
$\zeta_\mu$ for some $\mu\in \lambda+R(\n^-)$,
$\mu\neq\lambda$. But then $\bar\zeta_\mu\neq\bar\zeta_\lambda$. A contradiction.

Since $S$ is generic, the above argument implies
$(U(\g)\otimes_{U(\p^{-})}Q_{\chi}^{\s})^{\n^-}=Q_\chi^\s$.
\end{proof}

Now we can finish the proof of the theorem. By Lemma \ref{parind}
$$(U(\g)\otimes_{U(\p^{-})}Q_{\chi}^{\s})^{\m}=(Q_\chi^\s)^{\m^\s}=\bar W_\chi^\s.$$
That implies $d=0$.
\end{proof}

\subsection{The case of a regular $\chi$} If $\chi$ is regular and admits an even good $\Z$-grading, then
$\g$ is isomorphic to $\mathfrak{sl}(m|n),\mathfrak{osp}(2|2n)$ or $Q(n)$. In this case we set
$\displaystyle\p=\bigoplus_{i\geq 0}\g_i$.
If  $\g$ is of type II, then we define $\p$ as in Corollary \ref {specialparabolic}.

If $\g=Q(n)$ we set $k=\frac{n}{2}$ if $n$ is even and $\frac{n-1}{2}$
if $n$ is odd. In other cases we set $k=d$ (the defect of $\g$) if $\g$ is of type I or $\g$ is of
type II and $\hbox{dim}\g_{\bar 1}^\chi$ is even. If $\g$ is of type
II and  $\hbox{dim}\g_{\bar 1}^\chi$ is odd, then we set $k=d+1$.

\begin{proposition}\label{AL} $\bar W_\chi^\s$ satisfies Amitsur--Levitzki
identity, i.e. for any 
$u_1,\dots,u_{2^{k+1}}\in \bar W_\chi^\s$
\begin{equation}\label{equAL}
\sum_{\sigma\in S_{2^{k+1}}}sgn(\sigma)u_{\sigma(1)}\dots u_{\sigma{(2^{k+1})}}=0.
\end{equation}
\end{proposition}
\begin{proof} We first consider the case of even $\Z$-grading. Then $\bar W_\chi^\s=U(\s)$.
Let us assume first that $\g=Q(n)$. Then the even good $\Z$-grading
coincides with the Dynkin $\Z$-grading and 
$\s = \g_0=\h$ is a Cartan subalgebra of $\g$.
Denote
$$x_i = e_{i,i},\quad \xi_i = (-1)^{i+1}f_{i,i}.$$
Then $x_i$ lie in the center of $U(\h)$ and we have
$[f_{i,i},f_{i,i}]=2x_i$. From this it is easy to see that 
$U(\h_{\bar 0})=\mathbb C[x_1,\dots,x_n]$ coincides with the center of $U(\h)$.

Let $F$ denote the algebraic closure of the field of fractions of 
$U(\h_{\bar 0})$ and let $U(\h)_F=F\otimes_{U(\h_{\bar 0})}U(\h)$.
Then $U(\h)_F$ is isomorphic to the Clifford algebra associated with a
non-degenerate 
symmetric form on an $n$-dimensional space.
Thus, $U(\h)_F\simeq M_{2^k}(F)$ for even $n$ and $U(\h)_F\simeq
M_{2^k}(F)\times  M_{2^k}(F)$ for odd $n$, where by $M_s(F)$ 
we denote the
algebra of $s\times s$ matrices over $F$. Thus, by the Amitsur--Levitzki theorem (see \cite{AL}),
$U(\h)_F$ satisfies (\ref{equAL}). Since $U(\h)$ is a subalgebra of  $U(\h)_F$, it also satisfies (\ref{equAL}).

Now let $\g=\mathfrak{sl}(m|n)$ or $\mathfrak{osp}(2|2n)$.  Then
the even part of $\s$ coincides with the Cartan subalgebra $\h$, which
is abelian. The basis of the odd part consists of root 
elements
$X_1,\dots,X_{k}, Y_1,\dots,Y_k$ such that $[X_i,Y_j]=0$ if $i\neq j$,
$[X_i,X_j]=[Y_i,Y_j]=0$ for all $i,j\leq k$. 
Thus, $\s$ has a triangular decomposition
$\s=\s^-\oplus\h\oplus\s^+$, with $\s^+$ spanned by $X_1,\dots,X_k$
and $\s^-$ spanned by $Y_1,\dots,Y_k$. 
Let $\lambda\in\h^*$ and
$M_\lambda=U(\s)\otimes_{U(\h\oplus\s^+)}C_\lambda$ denote the Verma module over $\s$. The dimension of $M_\lambda$ equals $2^k$. An easy calculation shows that
$\displaystyle\prod_{\lambda\in\h^*}M_\lambda$ is a faithful $U(\s)$-module. Therefore $U(\s)$ is isomorphic to a subalgebra in
$\displaystyle\prod_{\lambda\in\h^*}\operatorname{End}_{\mathbb C} (M_\lambda)$. Since
$\displaystyle \prod_{\lambda\in\h^*}\operatorname{End}_{\mathbb C} (M_\lambda)$ satisfies the
Amitsur--Levitzki identity, $U(\s)$ must satisfy it as well.

Finally, let us consider the case when $\g$ is of type II. Here we are going to consider two subcases.
We will use notations of the proof of Corollary \ref{specialparabolic}.

First, let us assume that $\Pi$ contains an odd non-isotropic root
$\beta$. 
Then $\Pi'\setminus A=\{\beta_1,\dots,\beta_{k-1},\beta=\beta_{k}\}$. 
Then $[\s,\s]$ is a direct sum of $k-1$ copies of $\mathfrak{sl}(1|1)$
generated by the root spaces $\g_{\pm\beta_i}$, $i=1,\dots,k-1$ and
one copy of $\mathfrak{osp}(1|2)$ 
generated by  $\g_{\pm\beta_k}$.
Furthermore, $\m^\s\subset \mathfrak{osp}(1|2)$ is generated by $\g_{-2\beta}$.
Let us write $\s=\s'\oplus \r$, where $\r=\mathfrak {osp}(1|2)$. 
Then $\bar W_\chi^\s=U(\s')\otimes \bar W_\chi^\r$, where $\bar W_\chi^\r$ is the usual $W$-algebra
for the regular $\chi$ and $\r=\mathfrak {osp}(1|2)$.
In the following example we give an explicit description of $W$-algebra for $\mathfrak{osp}(1|2)$.

Let $\r = \o\s\p(1|2) = <X, Y, H \hbox{ } |\hbox{ } \theta, r>$, where
$$X = E_{2 3},
Y = E_{3 2},
H = E_{2 2}- E_{3 3},
\theta = E_{1 2} - E_{3 1},
r = E_{1 3} + E_{2 1}.$$
Let $\s\l(2) = <e, h, f>$, where $e = X, h = H, f = Y$.
The element $h$ defines a $\Z$-grading on $\r$:
$$\r = \r_{-2}\oplus \r_{-1}\oplus \r_{0}\oplus
\r_{1}\oplus \r_{2}, \hbox{ where }$$
$$\r_{-2} = <Y>, \quad \r_{-1} = <\theta>, \quad
\r_{0} = <H>,\quad
\r_{1} = <r>, \quad \r_{2} = <X>.$$
Consider the even non-degenerate invariant supersymmetric bilinear form
$(a|b) = {1\over 2} str(ab)$ on $\r$:
$(\theta | r) = 1, \quad (X| Y) = -{1\over 2}, \quad
(H|H) = -1.$ Let $\chi(x) = (e|x)$ for $x\in \r$ and let
$W_{\chi}$ be the corresponding $W$-algebra.
Note that $\g^{\chi} = \g^e = <X \hbox{ } |\hbox{ }r>$,
$\m = \r_{-2}$, and $\chi (Y)  = -{1\over 2}$.
We have that $\pi(\theta)\in W_{\chi}^{\r}$, and $\pi(\theta)^2 = {1\over 2}$.
Let $\Omega$ be the Casimir element of $\r$. Then
$$\pi(\Omega) =
\pi(2X + H -H^2 + 2r\theta).$$
Let
$$R = \pi(r - H\theta).$$
Note that $\pi(\Omega)$ and $R$ belong to $W_{\chi}$.

\begin{lemma}\label{osp}

a) $W_{\chi}$ is generated by $\pi(\Omega)$, $\pi(\theta)$ and $R$.
The defining relations are

\begin{eqnarray}
\begin{array}{llll}
&[\pi(\Omega), R] = [\pi(\Omega), \pi(\theta)] = 0,\nonumber\\
&[R, R] = \pi(\Omega), \quad [R, \pi(\theta)] = -{1\over 2}, \quad [\pi(\theta), \pi(\theta)] = 1.\nonumber
\end{array}
\end{eqnarray}

b) For any $c\in\C$, $W_\chi/(\Omega-c)$ is isomorphic to a Clifford
algebra with two generators and it has a unique irreducible representation $M_c$
of dimension $2$.
\end{lemma}
\begin{proof}
Since $\overline{\hbox{Gr}_K(\pi(\Omega))} = 2X$,
$\overline{\hbox{Gr}_K (R)} = r$, then (a) 
follows from Proposition \ref{towpr} (a).

The proof of (b) is straightforward.

\end{proof}

We use Lemma \ref{osp} (b) to prove the Amitsur--Levitzki identity in
the latter case. 
We again consider the family $M_\lambda\otimes M_c$, where 
$M_\lambda$ is the Verma module over $\s'$ and $M_c$ is as in Lemma \ref{osp} (b). Then    
$\displaystyle\prod_{\lambda\in\h^*,c\in\C}(M_\lambda\otimes M_c)$ is
a faithful $\bar W_\chi^\s$-module. Therefore  
$\bar W_\chi^\s$ is isomorphic to a subalgebra in
$\displaystyle\prod_{\lambda\in\h^*,c\in\C}\operatorname{End}_{\mathbb C} (M_\lambda\otimes M_c)$. Since
$\displaystyle \prod_{\lambda\in\h^*,c\in\C}\operatorname{End}_{\mathbb C} (M_\lambda\otimes M_c)$ satisfies the
Amitsur--Levitzki identity,  $\bar W_\chi^\s$ must satisfy it as well.

Finally we assume that all odd roots in $\Pi$ are isotropic. Then 
$$\Pi'\setminus A=\{\beta_1,\dots,\beta_{k-1},\beta_{k}\}$$ with the only non-orthogonal pair
$\beta_{k-1},\beta_k$. 
In this case $[\s,\s]$ is a direct sum of $k-2$ copies of $\mathfrak{sl}(1|1)$
generated by the root spaces $\g_{\pm\beta_i}$, $i=1,\dots,k-2$ and
one copy of $\mathfrak{sl}(1|2)$ generated by  
$\g_{\pm\beta_{k-1}},\g_{\pm\beta_k}$.
Furthermore, $\m^\s\subset \mathfrak{sl}(1|2)$ is generated by $f\in\g_{-\beta_{k-1}-\beta_k}$.
As in the previous case we write $\s=\s'\oplus \r$, where
$\r=\mathfrak {sl}(1|2)$. Then 
$\bar W_\chi^\s=U(\s')\otimes \bar W_\chi^\r$, where 
$$\bar W_\chi^\r=(U(\r)\otimes_{\C f}C_\chi)^f,$$
where $\chi(f)=1$. 

We realize $\r$ in the standard matrix form and introduce the following notations:

\begin{eqnarray}
\begin{array}{llll}
&h_1=E_{11}+E_{33}, h_2=E_{11}+E_{22}, f=E_{32}, e=E_{23}, e^+=E_{13}, e^-=E_{12},\\
& f^-=E_{31}, f^+=E_{21}, C=h_1+h_2.\nonumber
\end{array}
\end{eqnarray}
Let $\pi:U(\r)\to U(\r)/U(\r)(f-1)$ be the natural projection.
We denote by $\Omega$ the quadratic Casimir element of $\r$ and set
$$a=[f^-,e^+e^-]=h_1e^--e^+,\,\,b=[e^-,f^+f^-]=h_2f^--f^+.$$
The reader can easily check that $\pi(C),\pi(\Omega),\pi(e^-),\pi(f^-),\pi(a)$ and $\pi(b)$ belong to $\bar W_\chi^\r$.

\begin{lemma}\label{sk}

a) $\bar W^\r_{\chi}$ is generated by  $\pi(C),\pi(\Omega),\pi(e^-),\pi(f^-),\pi(a)$ and $\pi(b)$.
It is clear that $\pi(\Omega)$ lies in the center of  $\bar W^\r_{\chi}$. The other defining relations are

\begin{eqnarray}
\begin{array}{llll}
&[\pi(C), \pi(e^-)]=\pi(e^-), [\pi(C), \pi(a)]=\pi(a),\\
& [\pi(C), \pi(f^-)]=-\pi(f^-), [\pi(C), \pi(b)]=-\pi(b),\\
&[\pi(e^-),\pi(f^-)]=1, [\pi(a),\pi(b)]=\pi(\Omega),\nonumber
\end{array}
\end{eqnarray}
and the commuators of all other odd generators are zero.

b) Let $c,d\in\C$, $c\neq 0$, $U$ be the subalgebra in  
$\bar W_\chi^\r$ generated by $\pi(C),\pi(\Omega),\pi(a)$ and 
$\pi(e^-)$, 
$U_{c,d}=U/(\pi(\Omega)-c,\pi(C)-d,\pi(a),\pi(e^-))$ be the one-dimensional $U$-module. The induced module $M_{c,d}=\bar W_\chi^\r\otimes_UU_{c,d}$ is simple
and has dimension $4$. The product $\displaystyle\prod_{c,d\in \C}M_{c,d}$ is a faithful $\bar W_\chi^\r$-module.
\end{lemma}
\begin{proof} We leave the proof to the reader. For assertion (a) one should use a suitable modification of
Proposition \ref{towpr} (a).
\end{proof}

We also leave to the reader the proof of Proposition \ref{AL} in the
last case since it is 
completely similar to the previous case. 
\end{proof}

In what follows we denote by $\mathcal A$ the image $\vartheta(W_\chi)$ of  $W_\chi$ in  $\bar W_\chi^\s$.

\begin{corollary}\label{corAL} $W_\chi$ satisfies (\ref{equAL}).
\end{corollary}
\begin{proof} By Proposition \ref{HCinjective}, $\mathcal A\simeq W_\chi$. By Proposition \ref{AL}, 
$\mathcal A$ satisfies  (\ref{equAL}).
\end{proof}

\begin{proposition}\label{irred} Let $M$ be a simple $W_\chi$-module. Then $\dim M\leq 2^{k+1}$.
\end{proposition}
\begin{proof} Consider $M$ as a module over the associative algebra $W_\chi$, forgetting the $\mathbb Z_2$-grading.
Then either $M$ is simple or $M$ is a direct sum of two non-homogeneous simple submodules:
 $M = M_1\oplus M_2$.

In the former case we claim that  $\dim M\leq 2^{k}$.
Indeed, assume  $\dim M>2^{k}$. Let $V$ be a subspace of dimension $2^k+1$. By the density theorem for any $X_1,\dots,X_{2^{k+1}}\in\operatorname{End}_{\mathbb C}(V)$
one can find $u_1,\dots, u_{2^{k+1}}$ in $W_\chi$ such that
$(u_i)_{|V}=X_i$ for all $i=1,\dots, 2^{k+1}$. 
Since $\operatorname{End}_{\mathbb C}(V)$ does not satisfy
(\ref{equAL}), we obtain contradiction with Corollary \ref{corAL}.

In the latter case, we can prove in the same way that
$\dim M_1\leq 2^{k}$ and $\dim M_2\leq 2^{k}$. Therefore  $\dim M\leq 2^{k+1}$.

\end{proof}

\begin{conjecture} Every irreducible representation of $\mathcal
A\simeq W_\chi$ is isomorphic to a subquotient of some 
irreducible representation of $\bar W_\chi^\s$
restricted to $\mathcal A$.
\end{conjecture}

\section{Generators of $W_\chi$ for the queer Lie superalgebra $Q(n)$}

In the rest of the paper we study in detail the case when $\chi$ is regular and $\g=Q(n)$.

In this section we construct some generators of $W_\chi$. In particular, we will prove that $W_\chi$ is finitely generated.
We use the elements $e_{i,j}^{(m)}$ and $f_{i,j}^{(m)}$ of  $U(Q(n))$ defined in \cite{S} recursively:
\begin{align}\label{equG1}
& e_{i,j}^{(m)} = \sum_{k = 1}^n e_{i,k}e_{k,j}^{(m-1)} +
(-1)^{m+1}\sum_{k = 1}^n f_{i,k}f_{k,j}^{(m-1)},\\
&f_{i,j}^{(m)} = \sum_{k = 1}^n e_{i,k}f_{k,j}^{(m-1)} +
(-1)^{m+1}\sum_{k = 1}^n f_{i,k}e_{k,j}^{(m-1)}.
\nonumber
\end{align}
Then
\begin{align}\label{equG2}
&[e_{i,j}, e_{k,l}^{(m)}] = \delta_{j,k}e_{i,l}^{(m)} - \delta_{i,l}e_{k,j}^{(m)},\quad
[e_{i,j}, f_{k,l}^{(m)}] = \delta_{j,k}f_{i,l}^{(m)} - \delta_{i,l}f_{kj}^{(m)},\\
&[f_{i,j}, e_{k,l}^{(m)}] = (-1)^{m+1}\delta_{j,k}f_{i,l}^{(m)} - \delta_{i,l}f_{k,j}^{(m)},\nonumber\\
&[f_{i,j}, f_{k,l}^{(m)}] = (-1)^{m+1}\delta_{j,k}e_{i,l}^{(m)} + \delta_{i,l}e_{k,j}^{(m)}.
\nonumber
\end{align}
\begin{proposition} A. Sergeev \cite{S}.

\noindent
The elements $\sum_{i=1}^n e_{i,i}^{(2m+1)}$ generate $Z(Q(n))$.
\end{proposition}

\begin{remark}
In contrast with the Lie algebra case the center $Z(Q(n))$ is not Noetherian, in particular, it is not finitely generated.
\end{remark}

\begin{lemma}\label{inW} $\pi(e_{n,1}^{(m)})$ and $\pi(f_{n,1}^{(m)})$ belong to $W_\chi$.
\end{lemma}
\begin{proof}
By (\ref{equG2}) we have that
$$[e_{i,j},e_{n,1}^{(m)}]=[f_{i,j},e_{n,1}^{(m)}]=[e_{i,j},f_{n,1}^{(m)}]=
[f_{i,j},f_{n,1}^{(m)}]=0$$
for all $i>j$.
In other words, $e_{n,1}^{(m)},f_{n,1}^{(m)}\in U(\mathfrak g)^{ad \m}$.
Hence $\pi(e_{n,1}^{(m)}),\pi(f_{n,1}^{(m)})\in W_\chi$.
\end{proof}

\begin{lemma}\label{claim1}
Let $1\leq l\leq n-1$. Then
\begin{equation}\label{equG8}
\pi(e_{m,1}^{(l)})=\left\{ \begin{array}{cc}
 1 &   if \quad  m = l+1,\\
 0 &  if \quad  l+2\leq m \leq n,\\
\end{array}\right.\quad
\pi(f_{m,1}^{(l)})=0, \hbox{ if } l+1\leq m \leq n.
\end{equation}
\end{lemma}
\begin{proof}
We will prove the statement by induction in $l$. For $l = 1$ we have that
\begin{equation}\label{equG9}
\pi(e_{m,1}^{(1)})=\pi(e_{m,1}),\quad \pi(f_{m,1}^{(1)})=\pi(f_{m,1}).
\end{equation}
Then (\ref{equG8}) follows from  (\ref{equ2}).
Assume that (\ref{equG8}) holds for $l$.
From (\ref{equG1}) we have that
\begin{align}\label{equG10}
&e_{m,1}^{(l+1)} = \sum_{k = 1}^n e_{m,k}e_{k,1}^{(l)} +
(-1)^{l}\sum_{k = 1}^n f_{m,k}f_{k,1}^{(l)},\\
&f_{m,1}^{(l+1)} = \sum_{k = 1}^n e_{m,k}f_{k,1}^{(l)} +
(-1)^{l}\sum_{k = 1}^n f_{m,k}e_{k,1}^{(l)}.
\nonumber
\end{align}
Note that
\begin{align}\label{equG11}
&[e_{m,k}, e_{k,1}^{(l)}] =  e_{m,1}^{(l)},\quad
[e_{m,k}, f_{k,1}^{(l)}] =  f_{m,1}^{(l)},\\
&[f_{m,k}, e_{k,1}^{(l)}] =  (-1)^{l+1}f_{m,1}^{(l)},\quad
[f_{m,k}, f_{k,1}^{(l)}] =
(-1)^{l+1}e_{m,1}^{(l)}.
\nonumber
\end{align}
Hence
\begin{align}
& e_{m,1}^{(l+1)} = \sum_{k = 1}^{m-1}(e_{k,1}^{(l)}e_{m,k} + e_{m,1}^{(l)}) +
\sum_{k = m}^{n}e_{m,k}e_{k,1}^{(l)} +
(-1)^{l}\Big(\sum_{k = 1}^{m-1}(-f_{k,1}^{(l)}f_{m,k} + (-1)^{l+1}e_{m,1}^{(l)}) +
\sum_{k = m}^{n}f_{m,k}f_{k,1}^{(l)}\Big),\nonumber\\
&f_{m,1}^{(l+1)} = \sum_{k = 1}^{m-1}(f_{k,1}^{(l)}e_{m,k} + f_{m,1}^{(l)}) +
\sum_{k = m}^{n}e_{m,k}f_{k,1}^{(l)} +
(-1)^{l}\Big(\sum_{k = 1}^{m-1}(e_{k,1}^{(l)}f_{m,k} + (-1)^{l+1}f_{m,1}^{(l)}) +
\sum_{k = m}^{n}f_{m,k}e_{k,1}^{(l)}\Big).
\nonumber
\end{align}
Then
\begin{align}
&\pi(e_{m,1}^{(l+1)}) = \sum_{k = 1}^{m-1}\pi(e_{k,1}^{(l)})\pi(e_{m,k})  +
\sum_{k = m}^{n}(\pi(e_{m,k})\pi(e_{k,1}^{(l)}) + (-1)^{l}\pi(f_{m,k})\pi(f_{k,1}^{(l)})) + \nonumber\\
&(-1)^{l+1}\Big(\sum_{k = 1}^{m-1}\pi(f_{k,1}^{(l)})\pi(f_{m,k})\Big), \nonumber\\
&\pi(f_{m,1}^{(l+1)}) = \sum_{k = 1}^{m-1}\pi(f_{k,1}^{(l)})\pi(e_{m,k})  +
\sum_{k = m}^{n}(\pi(e_{m,k})\pi(f_{k,1}^{(l)}) + (-1)^{l}\pi(f_{m,k})\pi(e_{k,1}^{(l)}))+ \nonumber\\
&(-1)^{l}\Big(\sum_{k = 1}^{m-1}\pi(e_{k,1}^{(l)})\pi(f_{m,k})\Big).
\nonumber
\end{align}
Then by (\ref{equ2})
\begin{equation}\label{equG12}
\pi(e_{m,1}^{(l+1)}) = \pi(e_{m-1,1}^{(l)}) + \sum_{k = m}^{n}\Big(\pi(e_{m,k})\pi(e_{k,1}^{(l)}) +
(-1)^{l}\pi(f_{m,k})\pi(f_{k,1}^{(l)})\Big),
\end{equation}
\begin{equation}\label{equG13}
\pi(f_{m,1}^{(l+1)}) = \pi(f_{m-1,1}^{(l)}) + \sum_{k = m}^{n}\Big(\pi(e_{m,k})\pi(f_{k,1}^{(l)}) +
(-1)^{l}\pi(f_{m,k})\pi(e_{k,1}^{(l)})\Big).
\end{equation}
Let $m\geq l+2$. Then by induction hypothesis,
\begin{equation}\label{equG13}
\pi(e_{k,1}^{(l)}) = \pi(f_{k,1}^{(l)}) = 0 \hbox{ for }k = m, \ldots, n.
\end{equation}
If $m = l+2$, then $\pi(e_{m,1}^{(l+1)}) = \pi(e_{l+1,1}^{(l)}) = 1$, and
if $m \geq  l+3$, then $\pi(e_{m,1}^{(l+1)}) = \pi(e_{m-1,1}^{(l)}) = 0$.
Also, if $m\geq l+2$, then
$\pi(f_{m,1}^{(l+1)}) = \pi(f_{m-1,1}^{(l)}) = 0$.
Hence (\ref{equG8}) holds for $l+1$.
\end{proof}

\begin{corollary}\label{corclaim1}
\begin{equation}\label{equG7}
\pi(e_{n,1}^{(m)})=0 \hbox{ for }m\leq n-2, \quad \pi(e_{n,1}^{(n-1)})= 1;\quad \pi(f_{n,1}^{(m)})=0
\hbox{ for } m\leq n-1,
\end{equation}
\end{corollary}

\begin{lemma}\label{claim11}
$$\pi(e_{n,1}^{(n)})= \pi(z),\quad
\pi(f_{n,1}^{(n)})= \pi(H_0).$$
\end{lemma}
\begin{proof}
Let $1\leq m\leq n$. We will show that
\begin{equation}\label{equG15}
\pi(e_{m,1}^{(m)}) = \sum_{k = 1}^{m}\pi(e_{k,k}),\quad
\pi(f_{m,1}^{(m)}) = \sum_{k = 1}^{m}(-1)^{k-1}\pi(f_{k,k}).
\end{equation}
Again we proceed by induction on $m$. If $m = 1$, then (\ref{equG15}) obviously holds by (\ref{equG9}).
Assume that (\ref{equG15}) holds for $m$.
From (\ref{equG1}) and (\ref{equG2}) we have that
\begin{align}
& e_{m+1,1}^{(m+1)} = \sum_{k = 1}^{n}e_{m+1,k}e_{k,1}^{(m)}
+ (-1)^m \sum_{k = 1}^{n}f_{m+1,k}f_{k,1}^{(m)} =
\sum_{k = 1}^{m}\Big(e_{k,1}^{(m)}e_{m+1,k} + e_{m+1,1}^{(m)}\Big) +
\sum_{k = m+1}^{n}e_{m+1,k}e_{k,1}^{(m)} + \nonumber\\
&(-1)^{m}\Big(\sum_{k = 1}^{m}(-f_{k,1}^{(m)}f_{m+1,k} + (-1)^{m+1}e_{m+1,1}^{(m)}) +
\sum_{k = m+1}^{n}f_{m+1,k}f_{k,1}^{(m)}\Big),\nonumber\\
&f_{m+1,1}^{(m+1)} = \sum_{k = 1}^{n}e_{m+1,k}f_{k,1}^{(m)}
+ (-1)^m \sum_{k = 1}^{n}f_{m+1,k}e_{k,1}^{(m)} =
\sum_{k = 1}^{m}\Big(f_{k,1}^{(m)}e_{m+1,k} + f_{m+1,1}^{(m)}\Big) +
\sum_{k = m+1}^{n}e_{m+1,k}f_{k,1}^{(m)} +\nonumber\\
&(-1)^{m}\Big(\sum_{k = 1}^{m}(e_{k,1}^{(m)}f_{m+1,k} + (-1)^{m+1}f_{m+1,1}^{(m)}) +
\sum_{k = m+1}^{n}f_{m+1,k}e_{k,1}^{(m)}\Big).
\nonumber
\end{align}
Using (\ref{equ2}) and (\ref{equG8}) we obtain
\begin{align}
&\pi(e_{m+1,1}^{(m+1)}) =
\pi(e_{m,1}^{(m)}) + \pi(e_{m+1,m+1}),\nonumber\\
&\pi(f_{m+1,1}^{(m+1)}) =
 \pi(f_{m,1}^{(m)}) + (-1)^m\pi(f_{m+1,m+1}).
 \nonumber
\end{align}
By induction hypothesis we have
\begin{align}
&\pi(e_{m+1,1}^{(m+1)}) = \sum_{k = 1}^{m}\pi(e_{k,k}) + \pi(e_{m+1,m+1}) = \sum_{k = 1}^{m+1}\pi(e_{k,k}),\nonumber\\
&\pi(f_{m+1,1}^{(m+1)}) = \sum_{k = 1}^{m}(-1)^{k-1}\pi(f_{k,k}) + (-1)^m\pi(f_{m+1,m+1}) =
\sum_{k = 1}^{m+1}(-1)^{k-1}\pi(f_{k,k}).
 \nonumber
\end{align}
Hence
$\pi(e_{n,1}^{(n)}) = \pi(z)$ and
$\pi(f_{n,1}^{(n)}) = \pi(H_0)$.
\end{proof}

Consider the Kazhdan filtration on $U(\b)$. By definition, the graded algebra
$Gr_K U(\b)$ is isomorphic to $S(\b)$. Moreover, $Gr_K U(\b)\simeq S(\b)$ is a commutative graded ring, where
the grading is induced from the Dynkin $\Z$-grading of $\g$.
For any $X\in U(\b)$ let $Gr_K(X)$ denote the corresponding element in $Gr_K U(\b)$ and
$P(X)$ denote the highest weight component of $Gr_K(X)$ in the Dynkin $\Z$-grading.
For $X\in U(\b)$, we denote by $\hbox{deg}P(X)$ the Kazhdan degree of $Gr_K(X)$ and by $\hbox{wt}P(X)$ the weight of the highest weight component
of  $Gr_K(X)$.

\begin{lemma}\label{claim2} $P(\pi(e_{n,1}^{(n)}))=z$,
\begin{align}\label{equG16}
&P(\pi(e_{n,1}^{(n-1+k)}))=e^{k-1}, k=2,\dots,n,\\
&P(\pi(f_{n,1}^{(n-1+k)}))=H_{k-1},k=1,\dots,n.
 \nonumber
\end{align}
\end{lemma}
\begin{proof} We will prove a more general statement. We claim that for $0\leq l\leq n-1$ and $1\leq p\leq n$
\begin{align}\label{equG17}
&P(\pi(e_{p,1}^{(p+l)}))=\sum_{i= 1}^re_{i,i+l},\\
&P(\pi(f_{p,1}^{(p+l)}))=\sum_{i= 1}^r(-1)^{l+1-i}f_{i,i+l},\hbox{ }
r = \hbox{min}\lbrace p, n-l\rbrace.
 \nonumber
\end{align}
In particular,
\begin{align}
&\hbox{deg}P(\pi(e_{p,1}^{(p+l)})) = \hbox{deg}P(\pi(f_{p,1}^{(p+l)})) = 2l+2,\nonumber\\
&\hbox{wt}P(\pi(e_{p,1}^{(p+l)})) = \hbox{wt}P(\pi(f_{p,1}^{(p+l)})) = 2l.
 \nonumber
\end{align}
We proceed to the proof of (\ref{equG17}) by induction on $l$ and $p$.
Note that if $l = 0$, then (\ref{equG17}) holds for any $1\leq p\leq n$ by (\ref{equG15}).
Assume that
if $l\leq k-1$, then (\ref{equG17}) holds for any $1\leq p\leq n$.
Let $l = k$. Show that (\ref{equG17}) holds for $p = 1$.
Note that
$$e_{1,1}^{(1+k)} = \Big(\sum_{i = 1}^{n}e_{1,i}e_{i,1}^{(k)}\Big) +
(-1)^k\Big(\sum_{i = 1}^{n}f_{1,i}f_{i,1}^{(k)}\Big).$$
Let $X = \pi(e_{1,i}e_{i,1}^{(k)})$, $Y = \pi(f_{1,i}f_{i,1}^{(k)})$ where $i = 1, \ldots, k$.
Note that
\begin{align}
&\hbox{deg}P(\pi(e_{1,i})) = \hbox{deg}P(\pi(f_{1,i})) = 2i,\nonumber\\
&\hbox{wt}P(\pi(e_{1,i})) = \hbox{wt}P(\pi(f_{1,i})) = 2i-2.
 \nonumber
\end{align}
By induction hypothesis,
\begin{align}
&\hbox{deg}P(\pi(e_{1,i}^{(k)})) = \hbox{deg}P(\pi(f_{1,i}^{(k)})) = 2(k-i)+2,\nonumber\\
&\hbox{wt}P(\pi(e_{1,i}^{(k)})) = \hbox{wt}P(\pi(f_{1,i}^{(k)})) = 2(k-i).
 \nonumber
\end{align}
Then
\begin{align}
&\hbox{deg}P(X) = 2k+2, \quad \hbox{wt}P(X) = 2k-2,\nonumber\\
&\hbox{deg}P(Y) = 2k+2, \quad \hbox{wt}P(Y) = 2k-2.
\nonumber
\end{align}
Let $X = \pi(e_{1,k+1}e_{k+1,1}^{(k)})$. Then by (\ref{equG8}) $X  = \pi(e_{1,k+1}).$ Hence
$$\hbox{deg}P(X) = 2k+2, \quad \hbox{wt}P(X) = 2k.$$
Finally, by (\ref{equG8})
$$\pi(e_{k+i,1}^{(k)}) = 0\hbox{ for }i = 2, \ldots, n-k, \quad
\pi(f_{k+i,1}^{(k)}) = 0\hbox{ for }i = 1, \ldots, n-k.$$
Hence
$$P(\pi(e_{1,1}^{(1+k)})) = e_{1,k+1}.$$
Let $l = k$ and assume that (\ref{equG17}) holds for $p \leq m$.
Show that it holds for $p = m+1$. Note that
$$e_{m+1,1}^{(m+1+k)} = \Big(\sum_{i = 1}^{n}e_{m+1,i}e_{i,1}^{(m+k)}\Big) +
(-1)^{m+k}\Big(\sum_{i = 1}^{n}f_{m+1,i}f_{i,1}^{(m+k)}\Big).$$
Thus
\begin{align}
&\pi(e_{m+1,1}^{(m+1+k)}) =
\Big(\sum_{i = 1}^{m-1}\pi(e_{m+1,i}e_{i,1}^{(m+k)})\Big) + \pi(e_{m+1,m}e_{m,1}^{(m+k)})+ \nonumber\\
&\sum_{i = 1}^{k}\pi(e_{m+1,m+i}e_{m+i,1}^{(m+k)}) +
\pi(e_{m+1,m+k+1}e_{m+k+1,1}^{(m+k)})+
\sum_{i = 2}^{n-m-k}\pi(e_{m+1,m+k+i}e_{m+k+i,1}^{(m+k)})+ \nonumber\\
&(-1)^{m+k}\Big(\sum_{i = 1}^{m}\pi(f_{m+1,i}f_{i,1}^{(m+k)}) +
\sum_{i = 1}^{k}\pi(f_{m+1,m+i}f_{m+i,1}^{(m+k)}) + \nonumber\\
&\sum_{i = 1}^{n-m-k}\pi(f_{m+1,m+k+i}f_{m+k+i,1}^{(m+k)})\Big).
\nonumber
\end{align}
Let $X = \pi(e_{m+1,i}e_{i,1}^{(m+k)})$, where $i = 1, \ldots, m-1$, and
$Y = \pi(f_{m+1,i}f_{i,1}^{(m+k)})$, where $i = 1, \ldots, m$.
 Then by (\ref{equG2}) and (\ref{equ2})
 \begin{align}
&X = \pi(e_{i,1}^{(m+k)}e_{m+1,i} + e_{m+1,1}^{(m+k)}) =
\pi(e_{m+1,1}^{(m+k)}),\nonumber\\
&Y = \pi(-f_{i,1}^{(m+k)}f_{m+1,i} + (-1)^{m+k+1}e_{m+1,1}^{(m+k)}) =
\pi((-1)^{m+k+1}e_{m+1,1}^{(m+k)}).
\nonumber
\end{align}
By induction hypothesis
\begin{align}\label{equG18}
&\hbox{deg}P(X) = \hbox{deg}P(Y) = 2k, \\
&\hbox{wt}P(X) = \hbox{wt}P(Y) = 2k-2.
\nonumber
\end{align}
Let $X = \pi(e_{m+1,m}e_{m,1}^{(m+k)})$. Then by (\ref{equG2}) and (\ref{equ2})
$$X = \pi(e_{m,1}^{(m+k)}e_{m+1,m} + e_{m+1,1}^{(m+k)}) =
\pi(e_{m,1}^{(m+k)} + e_{m+1,1}^{(m+k)}).$$
By induction hypothesis
\begin{align}\label{equG19}
&\hbox{deg}P(\pi(e_{m+1,1}^{(m+k)})) = 2k, \\
&\hbox{wt}P(\pi(e_{m+1,1}^{(m+k)})) = 2k-2,
\nonumber
\end{align}
\begin{align}\label{equG20}
&\hbox{deg}P(\pi(e_{m,1}^{(m+k)})) = 2k+2, \\
&\hbox{wt}P(\pi(e_{m,1}^{(m+k)})) = 2k.
\nonumber
\end{align}
Let $X = \pi(e_{m+1,m+i}e_{m+i,1}^{(m+k)})$, $Y = \pi(f_{m+1,m+i}f_{m+i,1}^{(m+k)})$
for $i = 1, \ldots, k$.
Then by induction hypothesis
\begin{align}\label{equG21}
&\hbox{deg}P(X) = \hbox{deg}P(Y) =2k+2, \\
&\hbox{wt}P(X) = \hbox{wt}P(Y) = 2k-2.
\nonumber
\end{align}
Let $X = \pi(e_{m+1,m+k+1}e_{m+k+1,1}^{(m+k)})$. Hence by (\ref{equG8}) $X = \pi(e_{m+1,m+k+1})$. Then
\begin{align}\label{equG22}
&\hbox{deg}P(X) = 2k+2, \\
&\hbox{wt}P(X) = 2k.
\nonumber
\end{align}
Finally, by (\ref{equG8}) $\pi(e_{m+1,m+k+i}e_{m+k+i,1}^{(m+k)}) = 0$ for $i = 2,\ldots, n-m-k$ and

\noindent
$\pi(f_{m+1,m+k+i}f_{m+k+i,1}^{(m+k)}) = 0$ for $i = 1,\ldots,n-m-k$.
From (\ref{equG18})-(\ref{equG22}) one can see that the highest degree component in $\pi(e_{m+1,1}^{(m+1+k)})$ has degree $2k+2$,
and its highest weight component has weight $2k$. In fact, if $m\geq n-k$, then by (\ref{equG20})
this component is  $P(\pi(e_{m,1}^{(m+k)}))$. By induction hypothesis
$P(\pi(e_{m,1}^{(m+k)})) = \sum_{i= 1}^{n-k}e_{i,i+k}$.
If  $m < n-k$, then $P(\pi(e_{m,1}^{(m+k)})) = \sum_{i= 1}^{m}e_{i,i+k}$. Note that in this case
$\pi(e_{m+1,1}^{(m+1+k)})$ has
an additional element $\pi(e_{m+1,m+k+1})$
of degree $2k + 2$ and weight $2k$ according to (\ref{equG22}).
Clearly,  $P(\pi(e_{m+1,m+k+1})) = e_{m+1,m+k+1}$ and $P(\pi(e_{m,1}^{(m+k)})) +  P(\pi(e_{m+1,m+k+1}))  \not= 0$.
Hence
$$P(\pi(e_{m+1,1}^{(m+1+k)}))= P(\pi(e_{m,1}^{(m+k)})) +  P(\pi(e_{m+1,m+k+1})) = \sum_{i= 1}^{m+1}e_{i,i+k}.$$
Then in either case,
$$P(\pi(e_{m+1,1}^{(m+1+k)}))  = \sum_{i= 1}^{r}e_{i,i+k},\hbox{ where }
r = \hbox{min}\lbrace m+1, n-k\rbrace.$$
Thus if $0\leq l\leq n-1$ and $1\leq p\leq n$, then
$$P(\pi(e_{p,1}^{(p+l)}))=\sum_{i= 1}^re_{i,i+l},\hbox{ where }
r = \hbox{min}\lbrace p, n-l\rbrace.$$
Similarly, one can prove that
$$P(\pi(f_{p,1}^{(p+l)}))=\sum_{i= 1}^r(-1)^{l+1-i}f_{i,i+l},\hbox{ }
r = \hbox{min}\lbrace p, n-l\rbrace.$$
In particular, if $p = n$ and $l = k$, where $k = 0, \ldots, n-1$, we have
\begin{align}
&P(\pi(e_{n,1}^{(n+k)}))=\sum_{i= 1}^{n-k}e_{i,i+k} = e^k,\nonumber\\
&P(\pi(f_{n,1}^{(n+k)}))=\sum_{i= 1}^{n-k}(-1)^{k+1-i}f_{i,i+k} = H_k.
\nonumber
\end{align}
\end{proof}

\begin{proposition}\label{generators}
$\pi(e_{n,1}^{(m)})$ and $\pi(f_{n,1}^{(m)})$ for $m=n,\dots,2n-1$ generate  $W_\chi$.
\end{proposition}
\begin{proof} The statement follows from Lemma \ref{claim2} and Proposition \ref{towpr} (a).
\end{proof}

\begin{corollary}\label{cor}
Lemma \ref{claim2} and Proposition  \ref{towpr} (b) imply that Conjecture \ref{generalconjecture} is true for $\g = Q(n)$ and regular $\chi$.
\end{corollary}

\begin{corollary}\label{conjQ}
The natural homomorphism $U(\g)^{ad \m}\to W_{\chi}$ is surjective.
\end{corollary}
\begin{proof} Since $e_{n,1}^{(m)}, f_{n,1}^{(m)}\in U(\g)^{ad \m}$, the statement follows from Proposition \ref{generators}.
\end{proof}

\section{Further results about the structure of $W_\chi$ for $\g = Q(n)$}

\subsection{The Harish-Chandra homomorphism for  $Q(n)$}
Recall that for  $\g = Q(n)$ and regular $\chi$ we have $\p=\b$.
We study in detail the restriction of the Harish-Chandra homomorphism
$\vartheta: U(\b)\longrightarrow U(\h)$ to $W_\chi$.
We start with calculating the images of the generators.

\begin{proposition}\label{HCgenerators}
\begin{align}\label{equA2}
&\vartheta(\pi(e_{n,1}^{(n+k-1)})) = [\sum_{i_1\geq i_2\geq\ldots \geq i_k}
(x_{i_1} + (-1)^{k+1}\xi_{i_1}) \ldots (x_{i_{k-1}} - \xi_{i_{k-1}})(x_{i_{k}} + \xi_{i_{k}})]_{even},\\
&\vartheta(\pi(f_{n,1}^{(n+k-1)})) = [\sum_{i_1\geq i_2\geq\ldots \geq i_k}
(x_{i_1} + (-1)^{k+1}\xi_{i_1}) \ldots (x_{i_{k-1}} - \xi_{i_{k-1}})(x_{i_{k}} + \xi_{i_{k}})]_{odd}.
\nonumber
\end{align}

\end{proposition}

\begin{proof}
We will prove by induction on $l$ and $p$ that if $0\leq l\leq n-1$ and $1\leq p \leq n$ then
\begin{align}\label{equA5}
&\vartheta(\pi(e_{p,1}^{(p+l)})) + \vartheta(\pi(f_{p,1}^{(p+l)})) =\\
&\sum_{p\geq i_1\geq i_2\geq\ldots \geq i_k\geq 1}
(x_{i_1} + (-1)^{l}\xi_{i_1}) \ldots (x_{i_{l}} - \xi_{i_{l}})(x_{i_{l+1}} + \xi_{i_{l+1}}).
\nonumber
\end{align}
Note that if $l = 0$, then (\ref{equA5}) holds for any $1\leq p \leq n$ since by (\ref{equG15})
\begin{align}
&\vartheta(\pi(e_{p,1}^{(p)})) + \vartheta(\pi(f_{p,1}^{(p)})) = \nonumber\\
&\vartheta(\sum_{i= 1}^p\pi(e_{i,i})) +
\vartheta(\sum_{i= 1}^p(-1)^{i-1}\pi(f_{i,i})) = \sum_{i= 1}^p(x_i + \xi_i).
\nonumber
\end{align}
Assume that if $l\leq k-1$, then (\ref{equA5}) holds for any $1\leq p \leq n$.
Let $l = k$, show that (\ref{equA5}) holds for $p = 1$.
We  have
\begin{align}
&\vartheta(\pi(e_{1,1}^{(1+k)})) + \vartheta(\pi(f_{1,1}^{(1+k)})) = \nonumber\\
&\vartheta(\pi(e_{1,1})\pi(e_{1,1}^{(k)}) + (-1)^k\pi(f_{1,1})\pi(f_{1,1}^{(k)})) +
\vartheta(\pi(e_{1,1})\pi(f_{1,1}^{(k)}) + (-1)^k\pi(f_{1,1})\pi(e_{1,1}^{(k)}))= \nonumber\\
&(e_{1,1} + (-1)^kf_{1,1})(\vartheta(\pi(e_{1,1}^{(k)}) + \vartheta(\pi(f_{1,1}^{(k)})) = \nonumber\\
&(x_1 +  (-1)^k\xi_1)
\sum_{i_1 = i_2 =\ldots = i_k = 1}
(x_{i_1} + (-1)^{k-1}\xi_{i_1}) \ldots (x_{i_{k-1}} - \xi_{i_{k-1}})(x_{i_{k}} + \xi_{i_{k}}) = \nonumber\\
&\sum_{i_1 = i_2 =\ldots = i_{k+1} = 1}
(x_{i_1} + (-1)^{k}\xi_{i_1}) \ldots (x_{i_{k}} - \xi_{i_{k}})(x_{i_{k+1}} + \xi_{i_{k+1}}).
\nonumber
\end{align}
Let $l = k$ and assume that (\ref{equA5}) holds for $p\leq m$. Show that it holds for $p=  m+1$.
By induction hypothesis we have
\begin{align}
&\vartheta(\pi(e_{m+1,1}^{(m+1+k)})) + \vartheta(\pi(f_{m+1,1}^{(m+1+k)})) =
\vartheta(\pi(e_{m,1}^{(m+k)})) + \vartheta(\pi(f_{m,1}^{(m+k)})) + \nonumber\\
&(e_{m+1,m+1} + (-1)^{m+k}f_{m+1,m+1})
\vartheta(\pi(e_{m+1,1}^{(m+k)})) + \vartheta(\pi(f_{m+1,1}^{(m+k)})) = \nonumber\\
&\sum_{m\geq i_1\geq i_2\geq\ldots \geq i_{k+1}\geq 1}
(x_{i_1} + (-1)^{k}\xi_{i_1}) \ldots (x_{i_{k}} - \xi_{i_{k}})(x_{i_{k+1}} + \xi_{i_{k+1}}) + \nonumber\\
&(x_{m+1} + (-1)^{k}\xi_{m+1})
\sum_{m+1\geq i_1\geq i_2\geq\ldots \geq i_{k}\geq 1}
(x_{i_1} + (-1)^{k-1}\xi_{i_1}) \ldots (x_{i_{k-1}} - \xi_{i_{k-1}})(x_{i_{k}} + \xi_{i_{k}}) = \nonumber\\
&\sum_{m+1\geq i_1\geq i_2\geq\ldots \geq i_{k+1}\geq 1}
(x_{i_1} + (-1)^{k}\xi_{i_1}) \ldots (x_{i_{k}} - \xi_{i_{k}})(x_{i_{k+1}} + \xi_{i_{k+1}}).
\nonumber
\end{align}
Thus (\ref{equA5}) is proven. In particular, if $p = n$ we obtain (\ref{equA2}).
\end{proof}
\begin{proposition}\label{corhc}
\begin{equation}\label{equA3}
\pi(e_{n,1}^{(n+1)}) =\pi({1\over 2} \sum_{i = 1}^n e_{i,i}^2 + \sum_{i = 1}^{n-1}e_{i,i+1} + \sum_{i<j}(-1)^{i-j}f_{i,i}f_{j,j} + {1\over 2}z^2 - z),
\end{equation}
and
\begin{equation}\label{equA4}
\vartheta (\pi(e_{n,1}^{(n+1)}))={1\over 2} \sum_{i = 1}^n x_{i}^2 +
 \sum_{i<j}\xi_{i}\xi_{j} + {1\over 2}z^2 - z.
 \end{equation}
\end{proposition}
\begin{proof}
We will prove by induction on $m$ that for $1\leq m\leq n$
\begin{equation}\label{equA6}
\pi(e_{m,1}^{(m+1)}) =\pi({1\over 2} \sum_{i = 1}^m e_{i,i}^2 +
\sum_{i = 1}^{{min}(m, n-1)}e_{i,i+1} +
\sum_{1\leq i<j\leq m}(-1)^{i-j}f_{i,i}f_{j,j}
 + {1\over 2}(\sum_{i = 1}^m e_{i,i})^2 - \sum_{i = 1}^m e_{i,i}).
 \end{equation}
 If $m = 1$, then
 $$\pi(e_{1,1}^{(2)}) = \pi(e_{1,1}^{2} + e_{1,2} - f_{1,1}^{2}) =
 \pi(e_{1,1}^{2} + e_{1,2} - e_{1,1}).$$
 Assume that (\ref{equA6}) holds for $m$. By (\ref{equG1})
 $$e_{m+1,1}^{(m+2)} = \sum_{i = 1}^n e_{m+1,i}e_{i,1}^{(m+1)} -
(-1)^{m}\sum_{k = i}^n f_{m+1,i}f_{i,1}^{(m+1)}.$$
Then by (\ref{equ2}) and (\ref{equG8})
$$\pi(e_{m+1,1}^{(m+2)}) = \pi(e_{m,1}^{(m+1)}) +
\pi(e_{m+1,m+1})\pi(e_{m+1,1}^{(m+1)}) + \pi(e_{m+1,m+2})\pi(e_{m+2,1}^{(m+1)}) -
(-1)^{m}\pi(f_{m+1,m+1})\pi(f_{m+1,1}^{(m+1)}).$$
By induction hypothesis and using (\ref{equG15}) we have
\begin{align}
&\pi(e_{m+1,1}^{(m+2)}) = \pi\Big({1\over 2} \sum_{i = 1}^m e_{i,i}^2 +
\sum_{i = 1}^{{min}(m, n-1)}e_{i,i+1} +
\sum_{1\leq i<j\leq m}(-1)^{i-j}f_{i,i}f_{j,j}
 + {1\over 2}(\sum_{i = 1}^m e_{i,i})^2 - \sum_{i = 1}^m e_{i,i}\Big) +\nonumber\\
&\pi\Big(e_{m+1,m+1}(\sum_{i = 1}^{m+1} e_{i,i}) + e_{m+1,m+2} -
(-1)^{m}f_{m+1,m+1}(\sum_{i = 1}^{m+1}(-1)^{i-1} f_{i,i})\Big)=\nonumber\\
&\pi\Big({1\over 2} \sum_{i = 1}^{m+1} e_{i,i}^2 +
\sum_{i = 1}^{{min}(m+1, n-1)}e_{i,i+1} +
\sum_{1\leq i<j\leq m+1}(-1)^{i-j}f_{i,i}f_{j,j}
 + {1\over 2}(\sum_{i = 1}^{m+1} e_{i,i})^2 - \sum_{i = 1}^{m+1} e_{i,i}\Big).
 \nonumber
\end{align}
 Thus (\ref{equA6}) is proven. In particular, if $m = n$ we obtain (\ref{equA3}).
 Finally, applying $\vartheta$ to (\ref{equA3}) we obtain (\ref{equA4}).
\end{proof}

\subsection{On the center of  $W_{\chi}$}
Recall that we denote by $\mathcal A$ the image $\vartheta(W_\chi)$ of  $W_\chi$ in $U(\h)$.
Set $\mathcal A^0=\mathcal A\cap U(\h_{\bar 0})$.

\begin{lemma}\label{specialgen}
Define odd elements $\Phi_0, \ldots, \Phi_{n-1}$ of $W_{\chi}$ as follows:
\begin{align}
&\Phi_0 = \pi(f_{n,1}^{(n)}) = \pi(H_0),\nonumber\\
&\Phi_k = \Big(\frac{1}{2}\operatorname{ad}(\pi(e_{n,1}^{(n+1)}))\Big)^k (\Phi_0), \quad k = 1, \ldots , n-1.\nonumber
\end{align}
\noindent
Then
(a) $P(\Phi_k)=H_k,$

(b)
$$[\Phi_m, \Phi_p] = 0, \hbox{ if } m+p  \hbox{ is odd },$$

(c) there exist $z_0, z_2, \ldots \in \pi(Z(Q(n)))$ such that
$$[\Phi_m, \Phi_p] = (-1)^m z_{m+p} \hbox{ if } m+p  \hbox{ is even. }$$
\end{lemma}

\begin{proof} Let $X,Y\in W_{\chi}$.  To prove (a) observe that if $P(X),P(Y)\in \g^\chi$ and $[P(X),P(Y)]\neq 0$, then
$P([X,Y])=[P(X),P(Y)]$. Since $P(\pi(e_{n,1}^{(n+1)}))=e$ and $P(\Phi_0)=H_0$,
the statement follows from the relation
$$H_k= (\frac{1}{2}\operatorname{ad}(e))^k (H_0).$$
To prove (b) and (c) we use $\vartheta$.
We first notice that (\ref{equA4}) implies
$$\vartheta(\Phi_k)=\sum_{j=1}^n \phi^{(k)}_j\xi_j$$
for some polynomial $ \phi^{(k)}_j\in\mathbb C[x_1,\dots,x_n]$ of degree $k$.
Hence $[\vartheta(\Phi_m),\vartheta(\Phi_p)]\in \mathbb C[x_1,\dots,x_n]$.
Since $x_i$ lie in the center of $\h$, we get
\begin{align}
&[\vartheta(\Phi_{m+1}),\vartheta(\Phi_p)]=
\frac{1}{2}[[\vartheta(\pi(e_{n,1}^{(n+1)})),\vartheta(\Phi_m)],\theta(\Phi_p)]=\nonumber\\
&-\frac{1}{2}[\vartheta(\Phi_{m}),[\vartheta(\pi(e_{n,1}^{(n+1)})),\vartheta(\Phi_p)]]=
-[\vartheta(\Phi_{m}),\vartheta(\Phi_{p+1})].
\nonumber
\end{align}
Since $\vartheta$ is injective, that implies
$$[\Phi_p, \Phi_q]=(-1)^{r-p}[\Phi_r, \Phi_s],$$
if $p+q=r+s$.

In particular, if $p+q$ is odd we have
$$[\Phi_p, \Phi_q]=(-1)^{q-p}[\Phi_q, \Phi_p]=0.$$
This implies (b).

To prove (c)  we set
$$z_i = [\Phi_0, \Phi_i]\quad\hbox{ }\hbox{ }\hbox{ for even }i, \quad 0\leq i\leq n-1.$$
Since $\vartheta(z_i)\in\mathcal A^0$, (c) follows from Lemma \ref{a0}.
\end{proof}

\begin{lemma}\label{a0}  $\mathcal A^0=\vartheta(\pi(Z(\g)))$.
\end{lemma}
\begin{proof} It is not hard to see that the restriction of $\vartheta$ on $Z(\g)$ coincides with
the standard Harish-Chandra homomorphism.
Thus, from Sergeev's result, \cite{S1}, we know that $\vartheta(Z(\g))$ coincides with the space of symmetric polynomials $p$ in $x_1,\dots,x_n$ satisfying
the additional condition
\begin{equation}\label{supercent}
\frac{\partial p}{\partial x_i}-\frac{\partial p}{\partial x_j}\in(x_i+x_j)U(\h_{\bar 0})
\end{equation}
for all $i<j\leq n$.
In view of Proposition  \ref{HCinjective}, it is sufficient to prove that if $p\in \mathcal A^0$, then $p$ is symmetric and  satisfies
(\ref{supercent}).

First, we will prove the last assertion in the case when $n=2$. It follows from Lemma \ref{specialgen} and
Theorem \ref{HCgenerators} that $\mathcal A$ is generated by $z_0=2x_1+2x_2$, $\phi_0=\xi_1+\xi_2$, $\phi_1=x_2\xi_1-x_1\xi_2$ and
$z_1=-\vartheta(\pi(e_{2,1}^3))+{1\over 4}z_0^2 -{1\over 2}z_0=x_1x_2-\xi_1\xi_2$. By direct calculation we can check that
$$\phi_0^2={1\over 2}z_0,\phi_0\phi_1=-{1\over 2}z_0\xi_1\xi_2,\phi_1^2={1\over 2}z_0x_1x_2, [z_1,\phi_0]=-2\phi_1, [z_1,\phi_1]=2x_1x_2\phi_0.$$
Let $\mathcal A_0$ denote the even part of $\mathcal A$. The above relations imply that  $\mathcal A_0$ is a subring in
$\mathbb C[z_0,x_1x_2]\oplus \mathbb C[z_0,x_1x_2]\xi_1\xi_2$. Moreover, $\mathcal A_0/(z_0\mathcal A_0)=\mathbb C[z_1].$
Therefore $\mathcal A^0=\mathbb C\oplus z_0\mathbb C[z_0,x_1x_2]$, i.e. $\mathcal A^0$ consists of symmetric polynomials satisfying (\ref{supercent}).

Let $p=\vartheta(u)$ for some $u\in W_\chi\subset U(\b)$. Then we have
\begin{equation}\label{cond1}
\pi(\operatorname{ad}e_{i+1,i}(u))=0
\end{equation}
and
\begin{equation}\label{cond2}
\pi(\operatorname{ad}f_{i+1,i}(u))=0
\end{equation}
for all $i=1,\dots,n-1$.

Let $\s_i$ be the subalgebra in $\g$ generated by $e_{i,i+1}, e_{i,i},  e_{i+1,i}, e_{i+1,i+1}$, $f_{i,i+1}$, $f_{i,i}$,  $f_{i+1,i}$, $f_{i+1,i+1}$.
Clearly, $\s_i$ is isomorphic to $Q(2)$.
Note that the orthogonal compliment $\s_i^\perp$ (with respect to the invariant form) is $\operatorname{ad}\s_i$-invariant,
$\b\cap\s_i^\perp$ is a Lie subalgebra and, moreover,
$$\pi(\operatorname{ad}e_{i+1,i}(u))=0,\quad \pi(\operatorname{ad}f_{i+1,i}(u))=0$$
whenever $u\in U(\b\cap \s_i^\perp).$

Therefore any $u\in W_\chi$ satisfying (\ref{cond1}) and (\ref{cond2}) for a given $i$ can be written in the form
$u=\sum u_j v_j$
for some $u_j\in U(\s_i\cap\b)$ satisfying (\ref{cond1}) and (\ref{cond2}) and arbitrary $v_j\in U(\s_i^\perp\cap\b)$.

Thus, (\ref{cond1}) and (\ref{cond2}) can be checked locally for $\s_i$.
Indeed, if $\vartheta(u)\in U(\h_{\bar 0})$, then
$$\vartheta(u)=\sum \vartheta(u_j)\vartheta(v_j),$$ where
$\vartheta(v_j)\in U(\h_{\bar 0}\cap \s_i^\perp)=\mathbb C[x_1,\dots,x_{i-1},x_{i+2},\dots,x_n]$ and
$\vartheta(u_j)\in U(\h_{\bar 0}\cap \s_i)=\mathbb C[x_i,x_{i+1}]$.
Since we already know the result for $Q(2)$, we obtain $\vartheta(u_j)(x_{i+1},x_i)=\vartheta(u_j)(x_i,x_{i+1})$ and
$\frac{\partial \vartheta(u_j)}{\partial x_i}-\frac{\partial \vartheta(u_j)}{\partial x_{i+1}}\in(x_i+x_{i+1})U(\h_{\bar 0}\cap\s_i)$.
Therefore $\vartheta(u)$ is invariant under all adjacent transpositions and therefore is symmetric. Moreover,
$$\frac{\partial\vartheta(u)}{\partial x_i}-\frac{\partial\vartheta(u)}{\partial x_{i+1}}\in(x_i+x_{i+1})U(\h_{\bar 0}).$$
Since $\vartheta(u)$ is symmetric, the last condition implies (\ref{supercent})
for $\vartheta(u)$.
\end{proof}

Let $\phi_k:=\vartheta (\Phi_k)$. Consider $U(\h)$ as a free $U(\h_{\bar 0})$-module and let $V$ denote the free submodule generated by $\xi_1,\dots,\xi_n$.
Then $V$ is equipped with $U(\h_{\bar 0})$-valued bilinear symmetric form $B(x,y)=[x,y]$. If $\omega=\vartheta(\pi({1\over 2}e_{n,1}^{(n+1)}))$, then
$T=\operatorname{ad\omega}$ is an $U(\h_{\bar 0})$-linear operator. As
we have seen in the proof of Lemma \ref{specialgen}, $T$ is
skew-symmetric with respect to the form $B$, i.e.
$$B(Tv,w)+B(v,Tw)=0.$$
Furthermore, in these terms $\phi_k=T^{k}(\phi_0)$. The matrix of $T$ in the standard basis $\xi_1,\dots,\xi_n$ has $0$ on the diagonal and
\begin{equation}\label{equmatrix}
t_{ij}=\left\{ \begin{array}{cc}
 x_j &   if \quad  i<j,\\
 -x_j &  if \quad  i>j.\\
\end{array}\right.
\end{equation}

\begin{lemma}\label{charpol} The characteristic polynomial $\operatorname{det} (\lambda\operatorname{Id}-T)$ of $T$ equals
$$\lambda^n+\sigma_2\lambda^{n-2}+\dots+\sigma_{2k}\lambda^{n-2k},$$ where
$\sigma_r=\displaystyle\sum_{i_1<\dots<i_r}x_{i_1}\dots x_{i_r}$ are the elementary symmetric functions.
\end{lemma}
\begin{proof} Let
$$p_n(x_1,\dots,x_n;\lambda)=\operatorname{det} (\lambda\operatorname{Id}-T)=\lambda^n+\sum_{i=1}^n f_{n,i}(x_1,\dots,x_n)\lambda^{n-i}.$$
Note that $f_{n,i}(x_1,\dots,x_n)$ is a symmetric polynomial, since the substitutions $x_i\mapsto x_j,x_j\mapsto x_i$ preserves the determinant of
$\lambda\operatorname{Id}-T$. It is also easy to calculate that $\operatorname{det} T=x_1\dots x_n$ if $n$ is even. If $n$ is odd, then  $\operatorname{det} T=0$, since
$T$ is skew-symmetric with respect to $B$. Finally, if $x_n=0$ we have a relation
$$p_n(x_1,\dots,x_{n-1},0;\lambda)=\lambda p_{n-1}(x_1,\dots,x_{n-1};\lambda).$$
That implies
$$f_{n,i}(x_1,\dots,x_{n-1},0)=f_{n-1,i}(x_1,\dots,x_{n-1}),$$
for $i\leq n-1$. Since it is also easy to show that the degree of $f_{n,i}$ is $i$, we can finish the proof by induction in $n$.
\end{proof}

\begin{corollary}\label{roots} There exists $s=(s_1,\dots,s_n)\in\mathbb R_{>0}^n$ such that the specialization of $\operatorname{det} (\lambda\operatorname{Id}-T)$
at the point $x_1=s_1,\dots,x_n=s_n$ has distinct eigenvalues.
\end{corollary}

\begin{proof} Assume that $n=2k$ is even. Let $\hbox{Pol}_n^{ev}$ denote the set of monic even polynomials in $\C[\lambda]$ of degree $n$ and
$\hbox{Pol}_n^{ev,+}$ denote the subset of polynomials with real positive coefficients. Let $\varphi:\mathbb R_{>0}^n\to \hbox{Pol}_n^{ev,+}$ be the specialization map, i.e.
$\varphi(s)$ is the specialization of  $\operatorname{det} (\lambda\operatorname{Id}-T)$ at $s\in\mathbb R_{>0}^n$. From the above Lemma, $d\varphi(s)$ is
surjective for generic  $s\in\mathbb R_{>0}^n$. Therefore $\hbox{Im}\varphi$ contains a non-empty open subset in $\mathbb R_{>0}^n$.

Define the map $\rho:\C^n\to \hbox{Pol}_n^{ev}$  by the formula
$$\rho (t_1,\dots,t_k)=\prod_{i=1}^k(\lambda^2-t_i^2).$$
Obviously, $\rho$ is surjective. Set
$$\mathcal U=\{(t_1,\dots,t_k)\in\C^k \hbox{ }|\hbox{ } t_i\neq \pm t_j{}\,\, \text{for  all}\,\, i\neq j\}.$$
Then $\rho(\mathcal U)$ is Zariski open in $\hbox{Pol}_n^{ev}$. Therefore the intersection $\hbox{Pol}_n^{ev,+}\cap\rho(\mathcal U)$ is a non-empty Zariski open
subset in $\hbox{Pol}_n^{ev,+}$. Hence $\hbox{Pol}_n^{ev,+}\cap\rho(\mathcal U)$ is dense in $\hbox{Pol}_n^{ev,+}$ in the usual topology and the intersection
$\hbox{Im}\varphi\cap\rho(\mathcal U)$ is not empty.
This implies the statement for even $n$.

For odd $n$ the proof is similar and we leave it to the reader.
\end{proof}

\begin{lemma}\label{independence}  $\phi_0,\dots\phi_{n-1}$ are linearly independent over $U(\h_{\bar 0})$.
\end{lemma}
\begin{proof} For each $s=(s_1,\dots,s_n)\in\mathbb C^n$ consider the ideal $I_s=(x_1-s_1,\dots,x_n-s_n)\in U(\h_{\bar 0})$.
Let $V_s=V/I_sV$ and $T_s$, $B_s$ and $(\phi_i)_s$ denote the corresponding operator, form and vector in $V_s$.
It suffices to show that  $(\phi_0)_s,\dots(\phi_{n-1})_s$ are linearly independent for some $s$. By Corollary \ref{roots}
we can find $s\in\mathbb R_{>0}^n$ such that all eigenvalues of $T_s$ are distinct.
Let $v_1, \ldots, v_n$ denote an eigenbasis for $T_s$, and
let $H_s$ denote the Hermitian form such that
$H_s(\xi_i,\xi_j)=B_s(\xi_i,\xi_j)$ for all $i,j=1,\dots,n$. Then $H_s$ is positive definite and
$T_s$ is skew-hermitian with respect to $H_s$.
Hence all eigenvalues
of $T_s$ are purely imaginary and $H_s(v_i,v_j)=0$ if $i\neq j$. Let $(\phi_0)_s=\sum_{i=1}^n a_iv_i.$
Since all eigenvalues of $H_s$ are distinct and $(\phi_l)_s=T^l(\phi_0)_s$, linear independence of
$(\phi_0)_s,\dots(\phi_{n-1})_s$ is equivalent to the fact that $a_i$
are not zero for all $i=1,\dots n$. Assume that some $a_i=0$. Since
$a_i=\frac{H_s(v_i,(\phi_0)_s)}{H_s(v_i,v_i)}$, that implies $H_s(v_i,(\phi_0)_s)=0$. Let $v_i=t_1\xi_1+\dots+t_n\xi_n$. Then the last condition
implies $\sum_{i=1}^n s_it_i=0$. But then the first coordinate of
$T_sv_i$ equals $s_2t_2+\dots+s_nt_n=-s_1t_1$. Since $T_sv_i=a v_i$ for some purely imaginary $a$, we obtain $t_1=0$. Repeating this argument
we can prove by induction that all $t_i$ are zero and obtain a contradiction.
\end{proof}

\noindent
{\bf Problem.} Calculate  $\vartheta(\Phi_i)$ and $\vartheta (z_i)$.

\begin{lemma}\label{center} The centralizer of $\mathcal A$ in $U(\h)$ coincides with $U(\h_{\bar 0})$.
\end{lemma}
\begin{proof} Suppose that $u$ lies in the centralizer of $\mathcal A$. Recall that  $F$ denotes the field of fractions of $U(\h_{\bar 0})$. Then
since $U(\h)$ is a free $U(\h_{\bar 0})$-module, $U(\h)\subset U(\h)_F$. By Lemma \ref{independence}, $\mathcal A_F$ contains $\xi_1,\dots,\xi_n$.
Hence we have $[\xi_i,u]=0$ for all $i=1,\dots,n$. Therefore $u$ lies in the center of $U(\h)$, which coincides with $U(\h_{\bar 0})$.
\end{proof}

\begin{corollary}\label{cent1}  The center of $\mathcal A$ coincides with $\mathcal A^0$.
\end{corollary}

\noindent
Proposition \ref{HCinjective}, Lemma \ref{a0} and Corollary \ref{cent1} imply.
\begin{corollary}\label{cent2} The center of $W_{\chi}$ coincides with $\pi(Z(Q(n)))$.
\end{corollary}

\subsection{New generators and relations}
We will need the following realization of $Q(n)$
given by
M. Nazarov and S. Sergeev in \cite{NS}.
Let the indices $i, j$ run through $-n, \ldots, -1, 1, \ldots, n$.
Put $p(i) = 0$ if $i>0$ and $p(i) = 1$ if $i < 0$.
As a vector space $Q(n)$ is spanned by the elements
$$F_{ij} = E_{ij} +  E_{-i,-j}.$$
Note that $F_{-i,-j} = F_{ij}$.
The elements $F_{ij}$ with $i>0$ form a basis of $Q(n)$.

 For any indices $n\geq 1$ and $i, j = \pm 1, \ldots, \pm n$,
we denote by $F_{ij}^{(m)}$ the following element of $U(Q(n))$:
\begin{equation}\label{equG3}
F_{ij}^{(m)} = \sum_{k_1, \ldots, k_{m-1}}(-1)^{p(k_1) + \ldots + p(k_{m-1})}
F_{ik_1}F_{k_1k_2}\ldots F_{k_{m-2}k_{m-1}}F_{k_{m-1}j}.
\end{equation}
Note that
\begin{equation}\label{equG4}
F_{ij}^{(m)} = (-1)^{m-1}F_{-i,-j}^{(m)},
\end{equation}
\begin{equation}\label{equG5}
F_{ij}^{(m)} = e_{ij}^{(m)}, \quad\hbox{ for }i, j > 0,
\end{equation}
\begin{equation}\label{equG6}
F_{ij}^{(m)} = (-1)^{m+1}f_{-i,j}^{(m)}, \quad\hbox{ for }i <0, j > 0.
\end{equation}

\begin{proposition}\label{evencommute}
\noindent
For odd $k$ and $m$ we have
\begin{equation}\label{equA10}
[\pi(e_{n,1}^{(n+k)}), \pi(e_{n,1}^{(n+m)})] = 0.
\end{equation}
\end{proposition}

\begin{proof}
We prove the statement by induction on $l = k+m$. Obviously, if $l = 2$, then (\ref{equA10}) is true.
Assume that the statement is true for odd $k$ and $m$ such that $k+m\leq l-2$.
According to \cite{NS}
$$[F_{n,1}^{(m)}, F_{n,1}^{(k)}] = \sum_{r = 1}^{m-1}
[F_{n,1}^{(k+r-1)}, F_{n,1}^{(m-r)}] +
\sum_{r = 1}^{m-1}(-1)^r(F_{-n,1}^{(k+r-1)}F_{-n,1}^{(m-r)} -
F_{n,-1}^{(m-r)}F_{n,-1}^{(k+r-1)}).$$
Thus from (\ref{equG4}), (\ref{equG5}), (\ref{equG6}) we have
\begin{equation}\label{equA11}
[e_{n,1}^{(m)}, e_{n,1}^{(k)}] =
\sum_{r = 1}^{m-1}
[e_{n,1}^{(k+r-1)}, e_{n,1}^{(m-r)}] +
\sum_{r = 1}^{m-1}(-1)^{r+1}((-1)^{k+m}f_{n,1}^{(k+r-1)}f_{n,1}^{(m-r)} + f_{n,1}^{(m-r)}f_{n,1}^{(k+r-1)}).
\end{equation}
Furthermore, from \cite{NS}
\begin{align}
&[F_{-n,1}^{(m)}, F_{-n,1}^{(k)}] = \sum_{r = 1}^{m-1}
(F_{-n,1}^{(k+r-1)}F_{-n,1}^{(m-r)} - F_{-n,1}^{(m-r)}F_{-n,1}^{(k+r-1)}) \nonumber\\
&+ \sum_{r = 1}^{m-1}(-1)^{r+1}(F_{n,1}^{(k+r-1)}F_{n,1}^{(m-r)} -
F_{-n,-1}^{(m-r)}F_{-n,-1}^{(k+r-1)}).
\nonumber
\end{align}
Thus from (\ref{equG4}), (\ref{equG5}), (\ref{equG6}) we have
\begin{align}\label{equA12}
&[f_{n,1}^{(m)}, f_{n,1}^{(k)}] =
-\sum_{r = 1}^{m-1}
(f_{n,1}^{(k+r-1)}f_{n,1}^{(m-r)} - f_{n,1}^{(m-r)}f_{n,1}^{(k+r-1)})\\
&+\sum_{r = 1}^{m-1}(-1)^{r+1}((-1)^{k+m}e_{n,1}^{(k+r-1)}e_{n,1}^{(m-r)} + e_{n,1}^{(m-r)}e_{n,1}^{(k+r-1)}).
\nonumber
\end{align}
\begin{lemma}\label{claim}
For odd $m$ we have
$$[\pi(f_{n,1}^{(n)}), \pi(f_{n,1}^{(n+m)})] = 0.$$
\end{lemma}
\begin{proof}
From (\ref{equA12})
\begin{align}
&[\pi(f_{n,1}^{(n)}), \pi(f_{n,1}^{(n+m)})] =
-\sum_{r = 1}^{m-1}
\Big(\pi(f_{n,1}^{(n+m+r-1)})\pi(f_{n,1}^{(n-r)}) - \pi(f_{n,1}^{(n-r)})\pi(f_{n,1}^{(n+m+r-1)})\Big)\nonumber\\
&+\sum_{r = 1}^{m-1}(-1)^{r+1}\Big(-\pi(e_{n,1}^{(n+m+r-1)})\pi(e_{n,1}^{(n-r)}) + \pi(e_{n,1}^{(n-r)})\pi(e_{n,1}^{(n+m+r-1)})\Big).
\nonumber
\end{align}
Note that the first sum is zero, since $\pi(f_{n,1}^{(n-r)}) = 0$ for $r\geq 1$ by (\ref{equG7}),
and the second sum is also zero, since $\pi(e_{n,1}^{(n-1)}) = 1$ and
$\pi(e_{n,1}^{(n-r)}) = 0$ for $r\geq 2$ by (\ref{equG7}).
\end{proof}

\noindent
Let
\begin{equation}\label{equA13}
[e_{n,1}^{(n+k)}, e_{n,1}^{(n+m)}]^e =
[e_{n,1}^{(n+m)}, e_{n,1}^{(n+k-1)}] + [e_{n,1}^{(n+m+1)}, e_{n,1}^{(n+k-2)}] + \ldots
+ [e_{n,1}^{(n+m+k-2)}, e_{n,1}^{(n+1)}],
\end{equation}
\begin{equation}\label{equA14}
[e_{n,1}^{(n+k)}, e_{n,1}^{(n+m)}]^f =
[f_{n,1}^{(n+m)}, f_{n,1}^{(n+k-1)}] - [f_{n,1}^{(n+m+1)}, f_{n,1}^{(n+k-2)}] + \ldots
- [f_{n,1}^{(n+m+k-2)}, f_{n,1}^{(n+1)}].
\end{equation}
Then
$$[\pi(e_{n,1}^{(n+k)}), \pi(e_{n,1}^{(n+m)})] =
\pi([e_{n,1}^{(n+k)}, e_{n,1}^{(n+m)}]^e) + \pi([e_{n,1}^{(n+k)}, e_{n,1}^{(n+m)}]^f),$$
since by (\ref{equG7})
$$\pi(e_{n,1}^{(n)}) = \pi(z), \quad \pi(e_{n,1}^{(n-1)}) = 1, \quad
\pi(e_{n,1}^{(n-r)}) = 0 \hbox{ for } r\geq 2,\quad
\pi(f_{n,1}^{(n-r)}) = 0 \hbox{ for } r\geq 1,$$
and by Lemma \ref{claim}
$$[\pi(f_{n,1}^{(n+m+k-1)}), \pi(f_{n,1}^{(n)})] = 0,$$
since $m+k-1$ is odd.

\noindent
Note that each of the sums in (\ref{equA13}) and (\ref{equA14}) has $k-1$ terms, where $k-1$ is even.

Denote by $[e_{n,1}^{(m)}, e_{n,1}^{(k)}]_e$  and by $[e_{n,1}^{(m)}, e_{n,1}^{(k)}]_f$
 the first and the second sum in (\ref{equA11}), respectively. Thus
 $$[e_{n,1}^{(m)}, e_{n,1}^{(k)}] =
 [e_{n,1}^{(m)}, e_{n,1}^{(k)}]_e + [e_{n,1}^{(m)}, e_{n,1}^{(k)}]_f.$$
Also, denote by $[f_{n,1}^{(m)}, f_{n,1}^{(k)}]_f$  and by $[f_{n,1}^{(m)}, f_{n,1}^{(k)}]_e$
 the first and the second sum in (\ref{equA12}), respectively. Thus
 $$[f_{n,1}^{(m)}, f_{n,1}^{(k)}] =
 [f_{n,1}^{(m)}, f_{n,1}^{(k)}]_f + [f_{n,1}^{(m)}, f_{n,1}^{(k)}]_e.$$
Let
$$A^m = \pi([e_{n,1}^{(n+m)}, e_{n,1}^{(n+k-1)}] + [e_{n,1}^{(n+m+1)}, e_{n,1}^{(n+k-2)}]
+[f_{n,1}^{(n+m)}, f_{n,1}^{(n+k-1)}] - [f_{n,1}^{(n+m+1)}, f_{n,1}^{(n+k-2)}]).$$
We claim that $A^m = 0$. Note that $A^m = A^m_e + A^m_f$, where
$$A^m_e = \pi([e_{n,1}^{(n+m)}, e_{n,1}^{(n+k-1)}]_e + [e_{n,1}^{(n+m+1)}, e_{n,1}^{(n+k-2)}]_e +
[f_{n,1}^{(n+m)}, f_{n,1}^{(n+k-1)}]_e -[f_{n,1}^{(n+m+1)}, f_{n,1}^{(n+k-2)}]_e),$$
$$A^m_f = \pi([e_{n,1}^{(n+m)}, e_{n,1}^{(n+k-1)}]_f + [e_{n,1}^{(n+m+1)}, e_{n,1}^{(n+k-2)}]_f +
[f_{n,1}^{(n+m)}, f_{n,1}^{(n+k-1)}]_f -[f_{n,1}^{(n+m+1)}, f_{n,1}^{(n+k-2)}]_f).$$
Let us show that $A^m_e = 0$. Note that
\begin{align}
&[e_{n,1}^{(n+m)}, e_{n,1}^{(n+k-1)}]_e =
[e_{n,1}^{(n+k-1)}, e_{n,1}^{(n+m-1)}] +
[e_{n,1}^{(n+k)}, e_{n,1}^{(n+m-2)}]+\nonumber\\
&[e_{n,1}^{(n+k+1)}, e_{n,1}^{(n+m-3)}] +
[e_{n,1}^{(n+k+2)}, e_{n,1}^{(n+m-4)}] +\ldots,\nonumber\\
&[f_{n,1}^{(n+m)}, f_{n,1}^{(n+k-1)}]_e =
-[e_{n,1}^{(n+k-1)},e_{n,1}^{(n+m-1)}] + [e_{n,1}^{(n+k)},e_{n,1}^{(n+m-2)}]\nonumber\\
&-[e_{n,1}^{(n+k+1)},e_{n,1}^{(n+m-3)}]
+ [e_{n,1}^{(n+k+2)},e_{n,1}^{(n+m-4)}] +\ldots.
\nonumber
\end{align}
Thus
$$[e_{n,1}^{(n+m)}, e_{n,1}^{(n+k-1)}]_e + [f_{n,1}^{(n+m)}, f_{n,1}^{(n+k-1)}]_e  =
2[e_{n,1}^{(n+k)},e_{n,1}^{(n+m-2)}] + 2[e_{n,1}^{(n+k+2)},e_{n,1}^{(n+m-4)}] +\ldots.$$
By induction hypothesis
$$[\pi(e_{n,1}^{(n+k)}),\pi(e_{n,1}^{(n+m-2)})] = [\pi(e_{n,1}^{(n+k+2)}),\pi(e_{n,1}^{(n+m-4)})] = \ldots = 0$$
for positive $k, m-2, m-4, \ldots$
Note that (\ref{equA10}) also holds for odd $k, m$ such that
$k\leq -1$ or $m\leq -1$ by (\ref{equG7}). Similarly
$$[e_{n,1}^{(n+m+1)}, e_{n,1}^{(n+k-2)}]_e - [f_{n,1}^{(n+m+1)}, f_{n,1}^{(n+k-2)}]_e =
2[e_{n,1}^{(n+k-2)},e_{n,1}^{(n+m)}] +2[e_{n,1}^{(n+k)}, e_{n,1}^{(n+m-2)}] +\ldots = 0.$$
By induction hypothesis
$$[\pi(e_{n,1}^{(n+k-2)}),\pi(e_{n,1}^{(n+m)})] =
[\pi(e_{n,1}^{(n+k)}), \pi(e_{n,1}^{(n+m-2)})]  = \ldots = 0.$$
Hence $A^m_e = 0$.
Let us show that $A^m_f = 0$. Note that
\begin{align}
&[e_{n,1}^{(n+m)}, e_{n,1}^{(n+k-1)}]_f = - [e_{n,1}^{(n+k-1)}, e_{n,1}^{(n+m)}]_f = \nonumber\\
&(f_{n,1}^{(n+m)}f_{n,1}^{(n+k-2)} - f_{n,1}^{(n+k-2)}f_{n,1}^{(n+m)})
-(f_{n,1}^{(n+m+1)}f_{n,1}^{(n+k-3)} - f_{n,1}^{(n+k-3)}f_{n,1}^{(n+m+1)})+\ldots, \nonumber\\
&[e_{n,1}^{(n+m+1)}, e_{n,1}^{(n+k-2)}]_f = - [e_{n,1}^{(n+k-2)}, e_{n,1}^{(n+m+1)}]_f = \nonumber\\
&(f_{n,1}^{(n+m+1)}f_{n,1}^{(n+k-3)} - f_{n,1}^{(n+k-3)}f_{n,1}^{(n+m+1)})
-(f_{n,1}^{(n+m+2)}f_{n,1}^{(n+k-4)} - f_{n,1}^{(n+k-4)}f_{n,1}^{(n+m+2)})+\ldots, \nonumber\\
&[f_{n,1}^{(n+m)}, f_{n,1}^{(n+k-1)}]_f =  [f_{n,1}^{(n+k-1)}, f_{n,1}^{(n+m)}]_f = \nonumber\\
&-(f_{n,1}^{(n+m)}f_{n,1}^{(n+k-2)} - f_{n,1}^{(n+k-2)}f_{n,1}^{(n+m)})
-(f_{n,1}^{(n+m+1)}f_{n,1}^{(n+k-3)} - f_{n,1}^{(n+k-3)}f_{n,1}^{(n+m+1)}) - \ldots, \nonumber\\
&-[f_{n,1}^{(n+m+1)}, f_{n,1}^{(n+k-2)}]_f =  -[f_{n,1}^{(n+k-2)}, f_{n,1}^{(n+m+1)}]_f = \nonumber\\
&(f_{n,1}^{(n+m+1)}f_{n,1}^{(n+k-3)} - f_{n,1}^{(n+k-3)}f_{n,1}^{(n+m+1)})
+(f_{n,1}^{(n+m+2)}f_{n,1}^{(n+k-4)} - f_{n,1}^{(n+k-4)}f_{n,1}^{(n+m+2)}) + \ldots.
\nonumber
\end{align}

\noindent
In the sum of the right-hand sides of these equations all terms cancel out.
Hence
$A^m_f = 0$.
Then
$A^m = 0.$
Similarly,
$$A^{m+2} = A^{m+4} = \ldots = A^{m+k-3} = 0.$$
Then
$$[\pi(e_{n,1}^{(n+k)}), \pi(e_{n,1}^{(n+m)})] = \sum_{i = 0}^{{1\over 2}{(k-3)}}A^{m+2i} = 0.$$

\end{proof}

\noindent
We set
$$z_i = \pi(e_{n,1}^{(n+i)}) \quad\hbox{ for odd }i, \quad 1\leq i\leq n-1.$$

\begin{theorem}\label{Gen}
Elements $z_0,\dots,z_{n-1}$ are algebraically independent
in  $W_\chi$. Together with $\Phi_0,\dots,\Phi_{n-1}$ they form a complete
set of generators in $W_\chi$.
\end{theorem}
\begin{proof} By Lemma \ref{specialgen}, we have $P(\Phi_i)=H_i$ for
$i\leq n-1$,
$P(z_i)=e^i$ for even $0<i\leq n-1$ and $P(z_0)=z$. By Lemma \ref{claim2},
$P(z_i)=e^i$ for odd $i\leq n-1$. Therefore the second assertion follows from Proposition \ref{towpr}. The algebraic independence of
$z_0,\dots,z_{n-1}$ follows from algebraic independence of the corresponding elements in $S(\g^\chi)$.
\end{proof}

\begin{conjecture} Let $\g$ be a basic classical Lie superalgebra and
$\chi$ is regular. Then it is possible to find a set of generators of
$W_{\chi}$ such that even generators commute, and the commutators of odd generators
are in $\pi(Z(\g))$.
\end{conjecture}

\section{{\bf Super-Yangian of $\bf Q(n)$}}

Super-Yangian $Y(Q(n))$  was studied by M. Nazarov  and A. Sergeev \cite{NS}.
Recall that
$Y(Q(n))$ is the associative unital superalgebra over $\C$ with the countable set of generators
$$T_{ij}^{(m)}\hbox{ where }m = 1, 2, \ldots \hbox{ and }i, j = \pm 1, \pm 2, \ldots, \pm n.$$

\noindent
The $\Z_2$-grading of the algebra $Y(Q(n))$ is defined as follows:
$$p(T_{ij}^{(m)})  =
p(i) + p(j), \hbox{ where } p(i) = 0 \hbox{ if } i>0, \hbox{ and } p(i) = 1  \hbox{ if }i<0.$$
To write down defining relations for these generators we employ the formal series

\noindent
in $Y(Q(n))[[u^{-1}]]$:
\begin{equation}\label{equY1}
T_{i,j}(u) = \delta_{ij}\cdot 1 + T_{i,j}^{(1)}u^{-1} + T_{i,j}^{(2)}u^{-2} + \ldots.
\end{equation}
Then for all possible indices $i, j, k, l$ we have the relations

\begin{align}\label{equY2}
& (u^2 - v^2)[T_{i,j}(u), T_{k,l}(v)]\cdot (-1)^{p(i)p(k) + p(i)p(l)  +  p(k)p(l)}\\
& = (u + v)(T_{k,j}(u)T_{i,l}(v) - T_{k,j}(v)T_{i,l}(u))\nonumber\\
&-(u - v)(T_{-k,j}(u)T_{-i,l}(v) - T_{k,-j}(v)T_{i,-l}(u))\cdot (-1)^{p(k) +p(l)},
\nonumber
\end{align}
 where $v$ is a formal parameter independent of $u$, so that (\ref{equY2}) is an equality in the algebra of formal Laurent series in $u^{-1}, v^{-1}$ with coefficients in $Y(Q(n))$.

\noindent
 For all indices $i, j$ we also have the relations
 \begin{equation}\label{equY3}
 T_{i,j}(-u) = T_{-i,-j}(u).
 \end{equation}

\noindent
Note that the relations (\ref{equY2}) and (\ref{equY3}) are equivalent to the following defining relations:

 \begin{align}\label{equY4}
 & ([T_{i,j}^{(m+1)}, T_{k,l}^{(r-1)}] - [T_{i,j}^{(m-1)}, T_{k,l}^{(r+1)}])
 \cdot (-1)^{p(i)p(k) + p(i)p(l)  + p(k)p(l)}  = \\
 & T_{k,j}^{(m)}T_{i,l}^{(r-1)} + T_{k,j}^{(m-1)}T_{i,l}^{(r)} -
 T_{k,j}^{(r-1)}T_{i,l}^{(m)} - T_{k,j}^{(r)}T_{i,l}^{(m-1)}\nonumber\\
 & + (-1)^{p(k) + p(l)}(-T_{-k,j}^{(m)}T_{-i,l}^{(r-1)} + T_{-k,j}^{(m-1)}T_{-i,l}^{(r)} +
 T_{k,-j}^{(r-1)}T_{i,-l}^{(m)} - T_{k,-j}^{(r)}T_{i,-l}^{(m-1)}),
 \nonumber
\end{align}

\begin{equation}\label{equY5}
 T_{-i,-j}^{(m)} = (-1)^m T_{i,j}^{(m)},
\end{equation}
where $m, r = 1, \ldots$ and $T_{ij}^{(0)} = \delta_{ij}$.

\begin{theorem}\label{Yan}
There exists a surjective homomorphism:
$$\varphi: Y(Q(1))\longrightarrow W_{\chi}$$
defined as follows:
$$\varphi(T_{1,1}^{(k)}) = (-1)^k\pi(e_{n,1}^{(n+k-1)}),\quad
\varphi(T_{-1,1}^{(k)}) = (-1)^k\pi(f_{n,1}^{(n+k-1)}), \hbox{ for } k= 1, 2, \ldots.$$
\end{theorem}

\begin{proof}

Note that even and odd generators of $Y(Q(1))$ are $T_{1,1}^{(m)}$ and $T_{-1,1}^{(m)}$, respectively,
where $m = 1, 2, \ldots.$
We are going to check that the relations (\ref{equY4}) for generators
of $Y(Q(1))$ are preserved by $\varphi$. We separate this checking in the following three cases.

\noindent
{\it Case 1: Even generators.}
We want first to check that $\varphi$ preserves the relation

\begin{align}\label{evenyang}
& [T_{1,1}^{(m)}, T_{1,1}^{(p)}] - [T_{1,1}^{(m-2)}, T_{1,1}^{(p+2)}] = \\
& T_{1,1}^{(m-1)}T_{1,1}^{(p)} + T_{1,1}^{(m-2)}T_{1,1}^{(p+1)} -
 T_{1,1}^{(p)}T_{1,1}^{(m-1)} - T_{1,1}^{(p+1)}T_{1,1}^{(m-2)}+ \nonumber\\
& -T_{-1,1}^{(m-1)}T_{-1,1}^{(p)} + T_{-1,1}^{(m-2)}T_{-1,1}^{(p+1)} +
 (-1)^{m+p-1}T_{-1,1}^{(p)}T_{-1,1}^{(m-1)} -
 (-1)^{m+p-1}T_{-1,1}^{(p+1)}T_{-1,1}^{(m-2)}.
  \nonumber
\end{align}
 First, we will  prove the relation
\begin{align}\label{equY6}
& (-1)^{m+p}\Big([e_{n,1}^{(m+n-1)}, e_{n,1}^{(p+n-1)}] - [e_{n,1}^{(m+n-3)}, e_{n,1}^{(p+n+1)}]\Big) = \\
& (-1)^{m+p-1}\Big(e_{n,1}^{(m+n-2)}e_{n,1}^{(p+n-1)} + e_{n,1}^{(m+n-3)}e_{n,1}^{(p+n)}-
 e_{n,1}^{(p+n-1)}e_{n,1}^{(m+n-2)} - e_{n,1}^{(p+n)}e_{n,1}^{(m+n-3)}\Big) +\nonumber\\
& (-1)^{m+p-1}\Big(-f_{n,1}^{(m+n-2)}f_{n,1}^{(p+n-1)} + f_{n,1}^{(m+n-3)}f_{n,1}^{(p+n)}\Big)+ \nonumber\\
& f_{n,1}^{(p+n-1)}f_{n,1}^{(m+n-2)} - f_{n,1}^{(p+n)}f_{n,1}^{(m+n-3)}.
  \nonumber
\end{align}
 Note that
 \begin{equation}\label{equY7}
 [e_{n,1}^{(m+n-1)}, e_{n,1}^{(p+n-1)}]_e - [e_{n,1}^{(m+n-3)}, e_{n,1}^{(p+n+1)}]_e =
 \sum_{r = 1}^2[e_{n,1}^{(p+n+r-2)}, e_{n,1}^{(m+n-r-1)}],
 \end{equation}
 \begin{align}\label{equY8}
& [e_{n,1}^{(m+n-1)}, e_{n,1}^{(p+n-1)}]_f - [e_{n,1}^{(m+n-3)}, e_{n,1}^{(p+n+1)}]_f = \\
& \sum_{r = 1}^2(-1)^{r+1}((-1)^{m+p}f_{n,1}^{(p+n+r-2)}f_{n,1}^{(m+n-r-1)} +
f_{n,1}^{(m+n-r-1)}f_{n,1}^{(p+n+r-2)}).
  \nonumber
\end{align}
Multiplying the sum of equations (\ref{equY7}) and (\ref{equY8}) by $(-1)^{m+p}$, we obtain (\ref{equY6}).

By Lemma \ref{claim1}, the application of $\pi$ to (\ref{equY6}) implies
that $\varphi$ preserves the relation (\ref{evenyang}).

\noindent
{\it Case 2: Odd generators.}
Next we will check that $\varphi$ preserves the relation
\begin{align}
 & -([T_{-1,1}^{(m)}, T_{-1,1}^{(p)}] - [T_{-1,1}^{(m-2)}, T_{-1,1}^{(p+2)}]) = \nonumber\\
 & T_{-1,1}^{(m-1)}T_{-1,1}^{(p)} + T_{-1,1}^{(m-2)}T_{-1,1}^{(p+1)} -
 T_{-1,1}^{(p)}T_{-1,1}^{(m-1)} - T_{-1,1}^{(p+1)}T_{-1,1}^{(m-2)}+ \nonumber\\
 &T_{1,1}^{(m-1)}T_{1,1}^{(p)} - T_{1,1}^{(m-2)}T_{1,1}^{(p+1)} +
 (-1)^{m+p}T_{1,1}^{(p)}T_{1,1}^{(m-1)} - (-1)^{m+p}T_{1,1}^{(p+1)}T_{1,1}^{(m-2)}.
   \nonumber
\end{align}
We claim that the following relation holds
 \begin{align}\label{equY9}
& (-1)^{m+p-1}\Big([f_{n,1}^{(m+n-1)}, f_{n,1}^{(p+n-1)}] - [f_{n,1}^{(m+n-3)}, f_{n,1}^{(p+n+1)}]\Big) = \\
& (-1)^{m+p-1}\Big(f_{n,1}^{(m+n-2)}f_{n,1}^{(p+n-1)} + f_{n,1}^{(m+n-3)}f_{n,1}^{(p+n)}-
 f_{n,1}^{(p+n-1)}f_{n,1}^{(m+n-2)} - f_{n,1}^{(p+n)}f_{n,1}^{(m+n-3)}\Big) + \nonumber\\
 & (-1)^{m+p-1}\Big(e_{n,1}^{(m+n-2)}e_{n,1}^{(p+n-1)} - e_{n,1}^{(m+n-3)}e_{n,1}^{(p+n)}\Big)+ \nonumber\\
 & -e_{n,1}^{(p+n-1)}e_{n,1}^{(m+n-2)} + e_{n,1}^{(p+n)}e_{n,1}^{(m+n-3)}.
    \nonumber
\end{align}
Indeed, use
\begin{align}\label{equY10}
& [f_{n,1}^{(m+n-1)}, f_{n,1}^{(p+n-1)}]_f - [f_{n,1}^{(m+n-3)}, f_{n,1}^{(p+n+1)}]_f = \\
& -\sum_{r = 1}^2(f_{n,1}^{(p+n+r-2)}f_{n,1}^{(m+n-r-1)} - f_{n,1}^{(m+n-r-1)}f_{n,1}^{(p+n+r-2)}),
\nonumber
\end{align}
\begin{align}\label{equY11}
& [f_{n,1}^{(m+n-1)}, f_{n,1}^{(p+n-1)}]_e - [f_{n,1}^{(m+n-3)}, f_{n,1}^{(p+n+1)}]_e = \\
& \sum_{r = 1}^2(-1)^{r+1}((-1)^{m+p}e_{n,1}^{(p+n+r-2)}e_{n,1}^{(m+n-r-1)} +
e_{n,1}^{(m+n-r-1)}e_{n,1}^{(p+n+r-2)}).
\nonumber
\end{align}
Multiplying the sum of equations (\ref{equY10}) and (\ref{equY11}) by $(-1)^{m+p-1}$, we obtain (\ref{equY9}).
The end of the proof is as in the previous case.

\noindent
{\it Case 3: Even and odd generators.}
Finally, we will check that $\varphi$ preserves the relation
\begin{align}
& [T_{1,1}^{(m)}, T_{-1,1}^{(p)}] - [T_{1,1}^{(m-2)}, T_{-1,1}^{(p+2)}] = \nonumber\\
& T_{-1,1}^{(m-1)}T_{1,1}^{(p)} + T_{-1,1}^{(m-2)}T_{1,1}^{(p+1)} -
 T_{-1,1}^{(p)}T_{1,1}^{(m-1)} - T_{-1,1}^{(p+1)}T_{1,1}^{(m-2)}+ \nonumber\\
 & T_{1,1}^{(m-1)}T_{-1,1}^{(p)} - T_{1,1}^{(m-2)}T_{-1,1}^{(p+1)} +
 (-1)^{m+p}T_{1,1}^{(p)}T_{-1,1}^{(m-1)} - (-1)^{m+p}T_{1,1}^{(p+1)}T_{-1,1}^{(m-2)}.
 \nonumber
\end{align}
We claim that the following relation holds
\begin{align}\label{equY12}
& (-1)^{m+p}\Big([e_{n,1}^{(m+n-1)}, f_{n,1}^{(p+n-1)}] - [e_{n,1}^{(m+n-3)}, f_{n,1}^{(p+n+1)}]\Big) = \\
 & (-1)^{m+p-1}\Big(f_{n,1}^{(m+n-2)}e_{n,1}^{(p+n-1)} + f_{n,1}^{(m+n-3)}e_{n,1}^{(p+n)}-
 f_{n,1}^{(p+n-1)}e_{n,1}^{(m+n-2)} - f_{n,1}^{(p+n)}e_{n,1}^{(m+n-3)}\Big) + \nonumber\\
 &(-1)^{m+p-1}\Big(e_{n,1}^{(m+n-2)}f_{n,1}^{(p+n-1)} - e_{n,1}^{(m+n-3)}f_{n,1}^{(p+n)}\Big) \nonumber\\
 &-e_{n,1}^{(p+n-1)}f_{n,1}^{(m+n-2)} + e_{n,1}^{(p+n)}f_{n,1}^{(m+n-3)}.
  \nonumber
\end{align}
 According to \cite{NS}
 \begin{align}
&[F_{n,1}^{(m)}, F_{-n,1}^{(k)}] = \sum_{r = 1}^{m-1}
(F_{n,1}^{(k+r-1)}F_{-n,1}^{(m-r)} - F_{n,1}^{(m-r)}F_{-n,1}^{(k+r-1)}) + \nonumber\\
&\sum_{r = 1}^{m-1}(-1)^{r+1}(F_{-n,1}^{(k+r-1)}F_{n,1}^{(m-r)} +
(-1)^{m+k}F_{-n,1}^{(m-r)}F_{n,1}^{(k+r-1)}).
  \nonumber
\end{align}
Thus from (\ref{equG4}), (\ref{equG5}), (\ref{equG6}) we have
\begin{align}\label{equY13}
& [e_{n,1}^{(m)}, f_{n,1}^{(k)}] =
\sum_{r = 1}^{m-1}
(-1)^r((-1)^{m+k}e_{n,1}^{(k+r-1)}f_{n,1}^{(m-r)} + e_{n,1}^{(m-r)}f_{n,1}^{(k+r-1)}) + \\
&\sum_{r = 1}^{m-1}(f_{n,1}^{(k+r-1)}e_{n,1}^{(m-r)} - f_{n,1}^{(m-r)}e_{n,1}^{(k+r-1)}).
  \nonumber
\end{align}
We denote by $[e_{n,1}^{(m)}, f_{n,1}^{(k)}]_{ef}$  and by $[e_{n,1}^{(m)}, f_{n,1}^{(k)}]_{fe}$
 the first and the second sum in (\ref{equY13}), respectively, then
 $$[e_{n,1}^{(m)}, f_{n,1}^{(k)}] =
 [e_{n,1}^{(m)}, f_{n,1}^{(k)}]_{ef} + [e_{n,1}^{(m)}, f_{n,1}^{(k)}]_{fe}.$$
 Note that
 \begin{align}\label{equY14}
&  [e_{n,1}^{(m+n-1)}, f_{n,1}^{(p+n-1)}]_{ef} - [e_{n,1}^{(m+n-3)}, f_{n,1}^{(p+n+1)}]_{ef} =\\
 & \sum_{r = 1}^2((-1)^{m+p}e_{n,1}^{(p+n+r-2)}f_{n,1}^{(m+n-r-1)} + e_{n,1}^{(m+n-r-1)}f_{n,1}^{(p+n+r-2)}),
   \nonumber
\end{align}
  \begin{align}\label{equY15}
& [e_{n,1}^{(m+n-1)}, f_{n,1}^{(p+n-1)}]_{fe} - [e_{n,1}^{(m+n-3)}, f_{n,1}^{(p+n+1)}]_{fe} = \\
& \sum_{r = 1}^2(f_{n,1}^{(p+n+r-2)}e_{n,1}^{(m+n-r-1)} -
f_{n,1}^{(m+n-r-1)}e_{n,1}^{(p+n+r-2)}).
   \nonumber
\end{align}
Multiplying the sum of equations (\ref{equY14}) and (\ref{equY15})  by $(-1)^{m+p}$, we obtain (\ref{equY12}).
The proof can be finished by the same argument as in two previous cases.

\end{proof}

\font\red=cmbsy10
\def\~{\hbox{\red\char'0016}}



\vskip 0.1in


\begin{thebibliography}{10}




\bibitem{AL}
A. S. Amitsur, J. Levitzki,
Minimal identities for algebras,
{\it Proc. Amer. Math. Soc.} {\bf 1} (1950), 449--463.

\bibitem{BFR}
J. Balog, L. Feh{\'e}r, L. O'Raifeartaigh, P. Forg{\'a}cs and A. Wipf,
Toda theory and $W$-algebra from a gauged WZNW point of view,
{\it Ann. Physics} {\bf 203} (1990), 76--136.

\bibitem{BR}
C. Briot, E. Ragoucy,
{$W$-superalgebras as truncations of super-Yangians},
{\it J. Phys. A} {\bf 36} (2003), no. 4, 1057--1081.

\bibitem{B}
J. Brown,
Twisted Yangians and finite $W$-algebras,
{\it Transform. Groups} {\bf 14} (2009), 87--114.




\bibitem{BBG}
J. Brown, J. Brundan, S. Goodwin,
{Principal
$W$-algebras for $GL(m|n)$}, Algebra Numb. Theory 7 (2013), 1849--1882.

\bibitem{BK} J. Brundan, A. Kleshchev,
{Shifted Yangians and finite $W$-algebras}, {\it Adv. Math.} {\bf 200} (2006), 136-195,


\bibitem{DK}
A. De Sole and V. Kac,
{Finite vs affine $W$-algebras}, {\it Jpn. J. Math.} {\bf 1} (2006) 137--261.




\bibitem{FRTW1}
L. Feh{\'e}r, L. O'Raifeartaigh, P. Ruelle, I. Tsutsui, and A. Wipf,
{Generalized Toda theories and $W$-algebras associated with integral gradings},
{\it Ann. Physics}  {\bf 213} (1992) 1--20.

\bibitem{FRTW2}
L. Feh{\'e}r, L. O'Raifeartaigh, P. Ruelle, I. Tsutsui, and A. Wipf,
{On Hamiltonian reductions of the Wess-Zumino-Novikov-Witten theories},
{\it Phys. Rep.} {\bf 222} (1992) 1--64.

\bibitem{GS} C. Gruson, V. Serganova, 
Cohomology of generalized supergrassmannians and 
character formulae for basic classical Lie superalgebras, Proc. Lond. Math. Soc. (3) 101 (2010), no. 3, 852–892.

\bibitem{J} A.Joseph, {Kostant's problem, Goldie
rank and the Gelfand--Kirillov conjecture},
{\it Invent. Math.} {\bf 56} (1980), 191--213.

\bibitem{H}
C. Hoyt,
{ Good gradings of basic Lie superalgebras},
{\it Israel J. Math.} {\bf 192} (2012) 251--280.

\bibitem{K}
 V. G. Kac,
{ Lie superalgebras},
{\it Adv. Math.} {\bf 26} (1977) 8--96.

\bibitem{KW} V. G. Kac, M. Wakimoto,
Integrable highest weight modules over affine superalgebras and number
theory. Lie theory and geometry, 
415–456, Progr. Math., 123, Birkhäuser Boston, Boston, MA, 1994.

\bibitem{Ko}
B. Kostant,
{On Wittaker vectors and representation theory}, {\it Invent. Math.} {\bf 48} (1978) 101--184.

\bibitem{L1}
I. Losev,
{Finite $W$-algebras},
{\it Proceedings of the International Congress of Mathematicians.} Volume III, 1281--1307,
Hindustan Book Agency, New Delhi, 2010.
arXiv:1003.5811v1.

\bibitem{L2}
I. Losev,
{Quantized symplectic actions and $W$-algebras},
{\it J. Amer. Math. Soc.} {\bf 23} (2010) 35--59.



\bibitem{L3}
I. Losev,
{Finite-dimensional representations of $W$-algebras},
{\it Duke Math. J.} {\bf 159} (2011), 99--143.


\bibitem{M}
A. Molev,
 { Yangians and classical Lie algebras}, Mathematical Surveys and Monographs,
    {\bf 143}, {\it Amer. Math. Soc.}, Providence, RI, 2007.

\bibitem{N}
M. Nazarov,
{Yangian of the queer Lie superalgebra},
{\it Comm. Math. Phys.} {\bf 208} (1999) 195--223.


\bibitem{NS}
M. Nazarov, A. Sergeev,
{Centralizer construction of the Yangian of the queer Lie superalgebra},
Studies in Lie Theory, 417--441, {\it Progr. Math.} {\bf 243}, Birkh{\"a}user
Boston, Boston, MA, 2006.

\bibitem{PS1}
E. Poletaeva, V. Serganova,
{On finite W-algebras for Lie superalgebras
in the regular case},
In: {\it Lie Theory and Its Applications in Physics}, V. Dobrev (ed.),
IX International Workshop. 20-26 June 2011, Varna, Bulgaria.
Springer Proceedings in Mathematics and Statistics, Vol. {\bf 36} (2013) 487--497.

\bibitem{P}
E. Poletaeva,
{On Kostant's Theorem  for Lie superalgebras},
in M. Gorelik, P. Papi (eds.) {\it Advances in Lie Superalgebras},
Springer INdAM Series, Vol. {\bf 7} (2014) 167-180.


\bibitem{Pr1}
A. Premet,
{Special transverse slices and their enveloping algebras},
{\it Adv. Math.} {\bf 170} (2002) 1--55.

\bibitem{Pr2}
A. Premet,
{Enveloping algebras of Slodowy slices and the Joseph ideal},
{\it J. Eur. Math. Soc.}  {\bf 9} (2007) 487--543.

\bibitem{Pr3}
A. Premet,
{Primitive ideals, non-restricted representations and finite $W$-algebras},
{\it Mosc. Math. J.} {\bf 7} (2007) 743--762.

\bibitem{Pr4}
A. Premet, {Enveloping algebras of Slodowy slices and Goldie rank},
{\it Trans. groups}, {\bf 16} (2011) 857--888.

\bibitem{RS}
E. Ragoucy and P. Sorba,
Yangian realizations from finite $W$-algebras,
{\it Comm. Math. Phys.} {\bf 203} (1999) 551--572.

\bibitem{S}
A. Sergeev,
{The centre of enveloping algebra for Lie superalgebra $Q(n, \C)$},
{\it Lett. Math. Phys.} 7 (1983) 177--179.

\bibitem{S1}
A. Sergeev,
{The invariant polynomials on simple Lie superalgebras},
{\it Represent. Theory} 3 (1999) 250--280 (electronic).



\bibitem{W}
W. Wang,
{Nilpotent orbits and finite $W$-algebras},
Geometric representation theory and extended affine Lie algebras, 71--105,
{\it Fields Inst. Commun.}  {\bf 59},
Amer. Math. Soc., Providence, RI,
2011; arXiv:0912.0689v2.

\bibitem{Z}
L. Zhao,
{Finite $W$-superalgebras for queer Lie superalgebras}.
\noindent
arXiv:1012.2326v2.



\end{thebibliography}
\end{document}